\documentclass[final, reqno, 11pt]{amsart}
\usepackage{amsmath}
\usepackage{amsfonts}
\usepackage{amssymb}
\usepackage{amsthm}

\newtheorem{theorem}{Theorem}[section]
\newtheorem{lemma}[theorem]{Lemma}
\newtheorem{proposition}[theorem]{Proposition}
\newtheorem{corollary}[theorem]{Corollary}
\theoremstyle{remark}
\newtheorem{observation}[theorem]{Remark}
\newtheorem{definition}[theorem]{Definition}

\newcommand{\dl}{\nabla}
\newcommand{\les}{\lesssim}

\newcommand{\mc}{\mathcal}
\newcommand{\be}{\begin{equation}}
\newcommand{\ee}{\end{equation}}
\newcommand{\ba}{\begin{array}}
\newcommand{\ds}{\displaystyle}

\newcommand{\ea}{\end{array}}
\newcommand{\bpm}{\begin{pmatrix}}
\newcommand{\epm}{\end{pmatrix}}
\newcommand{\lb}{\label}

\DeclareMathOperator{\supp}{supp}

\newcommand{\ov}{\overline}
\newcommand{\dd}{{\,}{d}}

\newcommand{\R}{\mathbb R}

\newcommand{\Z}{\mathbb Z}
\newcommand{\N}{\mathbb N}
\newcommand{\B}{\mathcal B}

\title[Supercritical Wave Equation]{Large Outgoing Solutions to Supercritical Wave Equations}
\author{Marius Beceanu}
\address{University at Albany SUNY, Department of Mathematics and Statistics, Earth Science 110, Albany, NY, 12222, USA}
\email{mbeceanu@albany.edu}
\author{Avy Soffer}
\address{Rutgers University Department of Mathematics, 110 Frelinghuysen Rd., Piscataway, NJ, 08854, USA}
\email{soffer@math.rutgers.edu}
\begin{document}
\maketitle
\numberwithin{equation}{section}
\begin{abstract} We prove the existence of global solutions to the energy-supercritical wave equation in $\R^{3+1}$
$$
u_{tt}-\Delta u \pm |u|^N u = 0,\ u(0) = u_0,\ u_t(0) = u_1, 4<N<\infty,
$$
for a large class of radially symmetric finite-energy initial data.\\
Functions in this class are characterized as being outgoing under the linear flow --- for a specific meaning of ``outgoing" defined below.\\
In particular, we construct global solutions for initial data with large (even infinite) critical Sobolev, Besov, Lebesgue, and Lorentz norms and several other large critical norms.
\end{abstract}

\tableofcontents
\section{Introduction}
\subsection{Statement of the main results}
Consider the semilinear wave equation in $\R^{3+1}$
\be\lb{eq_sup}
u_{tt}-\Delta u \pm |u|^N u = 0,\ u(0) = u_0,\ u_t(0) = u_1.
\ee
The equation is called focusing or defocusing according to whether the sign of the nonlinearity is $-$ or $+$.

For $\alpha \in (0, \infty)$, this equation is invariant under the scaling symmetries
$$
u(x, t) \mapsto \alpha^{2/N} u(\alpha x, \alpha t),
$$
as well as under the Lorentz group of transformations. Restricted to the initial data, the rescaling is
\be\lb{rescaling}
(u_0(x), u_1(x)) \mapsto (\alpha^{2/N} u_0(\alpha x), \alpha^{1+2/N} u_1(\alpha x)).
\ee
For $s_c =3/2-2/N$, the $\dot H^{s_c} \times \dot H^{s_c-1}$ Sobolev norm is invariant under the rescaling (\ref{rescaling}), making it the critical Sobolev norm for the equation. Equation (\ref{eq_sup}) is locally well-posed in the $\dot H^{s_c} \times \dot H^{s_c-1}$ norm. Note that the corresponding (critical) Lebesgue norm is $u_0 \in L^{p_c}$ with $p_c=3N/2$.

All non-critical norms of the solution can be made arbitrarily large or small by rescaling, but critical norms remain constant after rescaling.

An important conserved quantity for equation (\ref{eq_sup}) is energy, defined as
$$
E[u]:= \int_{\R^3\times\{t\}} \frac 1 2 |u_t(x, t)|^2 + \frac 1 2 |\dl u(x, t)|^2 \pm \frac 1 {N+2} |u(x, t)|^{N+2} \dd x.
$$
In case the equation is defocusing, energy controls the $\dot H^1 \times L^2$ norm of the solution, also called the energy norm.

Equation (\ref{eq_sup}) is energy-supercritical (or, in brief, supercritical) if $N>4$. 
The difficulty of the initial-value problem in this case lies in the fact that solutions cannot be controlled in the energy norm (as $s_c>1$) and no higher-level conserved quantities can be used either.

By the standard local existence theory, based on Strichartz estimates, any initial data in the critical Sobolev space $\dot H^{s_c} \times \dot H^{s_c-1}$ produce a solution, locally in time. If the initial data are sufficiently small in the critical Sobolev norm, then the corresponding solution exists globally in time and disperses, meaning that, for example, it has finite $L^{2N}_{t, x}$ Strichartz norm (the endpoints are $L^\infty_t L^{3N/2}_x$, which is not dispersive, and $L^{N/2}_t L^\infty_x$, which is achieved for $N>4$ or $N=4$ and radially symmetric solutions). In general, solutions with finite $L^{2N}_{t, x}$ norm preserve regularity (if $(u_0, u_1) \in \dot H^s \times \dot H^{s-1}$ for some $s\geq 1$, the solution remains in this space for its whole interval of existence), are stable under small perturbations, and can be continued for as long as the $L^{2N}_{t, x}$ norm remains finite.

In this paper we heavily use the reversed Strichartz inequalities introduced in \cite{becgol}, Lorentz and Besov spaces, and real and complex interpolation techniques. A good reference for the latter is \cite{bergh}. The main new technique is a decomposition of solutions to the free wave equation into outgoing and incoming components by means of orthogonal projections; see below.

We only consider the case of radially symmetric, i.e.\ rotation-invariant, solutions (but see the Appendix for a very different result). We also assume all solutions are real-valued.

We define radial \emph{outgoing} functions as follows:
\begin{definition}\lb{def_outgoing} A pair $(u_0, u_1)$ of radially symmetric functions or distributions is called outgoing if $u_1 = -(u_0)_r - \frac {u_0} r$.
\end{definition}
Since $-\partial_r-1/r$ and $(-\partial_r-1/r)^*$ are bounded from $\dot H^s$ to $\dot H^{s-1}$, $1 \leq s < 3/2$, it follows that $-\partial_r-1/r \in \B(\dot H^s, \dot H^{s-1})$ for $-1/2<s<3/2$. Thus, the above definition makes sense for $(u_0, u_1) \in \dot H^s \times \dot H^{s-1}$ for $-1/2 < s < 3/2$ (but not only). Also, $u_0$ completely determines $u_1$. See Section \ref{outgoing_incoming} and Definition~\ref{def_1} for more details.

For simplicity we suppose that $N$ is an integer in (\ref{eq_sup}). This makes little actual difference in the proof.

Our first result is an existence result for the class of initial data
$$\begin{aligned}
((\dot H^1 \cap L^\infty) \times L^2)_{out} + \dot H^{s_c} \times \dot H^{s_c-1} := \{(u_0, u_1) = (v_0, v_1) + (w_0, w_1) \mid \\
(v_0, v_1) \in ((\dot H^1 \cap L^\infty) \times L^2)_{out},\ (w_0, w_1) \in \dot H^{s_c} \times \dot H^{s_c-1}\},
\end{aligned}$$
where $((\dot H^1 \cap L^\infty) \times L^2)_{out}$ means radial and outgoing following Definition \ref{def_outgoing}.  For data in this class the outgoing component is in a weaker space than $\dot H^{s_c} \times \dot H^{s_c-1}$, but the incoming component must still be in $\dot H^{s_c} \times \dot H^{s_c-1}$.

\begin{theorem}\lb{thm_main}
Assume that $N \in (4, 12]$, the initial data $(u_0, u_1)=(v_0, v_1)+(w_0, w_1)$ decompose into a radial and outgoing component  $(v_0, v_1)$ and a second radial component $(w_0, w_1) \in \dot H^{s_c} \times \dot H^{s_c-1}$ such that
$$
\|v_0\|_{\dot H^1}^{4/N} \|v_0\|_{L^\infty}^{1-4/N} + \|(w_0, w_1)\|_{\dot H^{s_c} \times \dot H^{s_c-1}} << 1
$$
is sufficiently small. Then the corresponding solution $u$ to (\ref{eq_sup}) exists globally, forward in time, remains small in $((\dot H^1_x \cap L^\infty_x) \times L^2_x)_{out} + \dot H^{s_c}_x \times \dot H^{s_c-1}_x$, and disperses:
\be\lb{disp_est}\begin{aligned}
\|u\|_{L^{N/2}_t L^\infty_x} \les \|v_0\|_{\dot H^1}^{4/N} \|v_0\|_{L^\infty}^{1-4/N} + \|(w_0, w_1)\|_{\dot H^{s_c} \times \dot H^{s_c-1}}.
\end{aligned}\ee
In addition $u$ scatters: there exist $(w_{0+}, w_{1+}) \in \dot H^{s_c} \times \dot H^{s_c-1}$ such that
\be\lb{scatter}
\lim_{t \to \infty} \|(u(t), u_t(t))-\Phi(t)(v_0, v_1)-\Phi(t)(w_{0+}, w_{1+})\|_{\dot H^{s_c} \times \dot H^{s_c-1}} = 0.
\ee
Here $\Phi(t)$ is the flow induced by the linear wave equation.

If $(v_0, v_1) \in ((\dot H^1 \cap L^\infty) \times L^2)_{out}$ and $(w_0, w_1) \in \dot H^{s_c} \times \dot H^{s_c-1}$ are not small, then there exist an interval $I=[0, T]$ with $T > 0$ and a solution $u$ to (\ref{eq_sup}) defined on $\R^3 \times I$, with $(u_0, u_1)$ as initial data, such that $(u(t), u_t(t)) \in ((\dot H^1_x \cap L^\infty_x) \times L^2_x)_{out} + \dot H^{s_c}_x \times \dot H^{s_c-1}_x$ for $t \in I$ and
$$
\|u\|_{L^{N/2}_t L^\infty_x(\R^3 \times I)} < \infty.
$$
\end{theorem}

The case $N=4$ corresponds to $s_c=1$ and $N=12$ corresponds to~$s_c=~4/3$. The conclusion is still true, but trivial, when $N=4$.

These initial data are small in the critical $L^{p_c}$ norm. However, equation (\ref{eq_sup}) is not well-posed in $L^{p_c}$. On the other hand, these are arbitrarily large initial data, as measured in the critical Sobolev norm, which is the natural norm for this equation.

Dropping the scaling invariance, we can obtain a local existence result for large $((\dot H^1 \cap L^\infty) \times L^2)_{out}$ initial data in the subcritical sense (i.e.\ where the time of existence only depends on the size of the initial data). We can also obtain a global existence result for small initial data, such that the solution remains bounded in $((\dot H^1 \cap L^\infty) \times L^2)_{out} + \dot H^2 \cap \dot H^1 \times \dot H^1 \cap L^2$ for all times, i.e.\ the incoming component of the solution gains a full derivative. See Proposition \ref{local2} for both results.

As a consequence of Theorem \ref{thm_main}, outgoing and radial finite energy initial data of any size lead to a global solution forward in time if they are supported sufficiently far away from the origin.
\begin{corollary}\lb{cor_sup}
Assume that $N \in (4, 12]$, the initial data $(u_0, u_1)$ are radial, outgoing according to Definition \ref{def_outgoing}, supported outside the sphere $B(0, R)$, and
\be\lb{cond}
\|u_0\|_{\dot H^1}^2\, R^{4/N-1} << 1
\ee
is sufficiently small. Then the corresponding solution $u$ to (\ref{eq_sup}) exists globally, forward in time, and disperses: $\|u\|_{L^{N/2}_t L^\infty_x} \les \|u_0\|_{\dot H^1}\, R^{2/N-1/2}$.
\end{corollary}

Again, one can add a small $\dot H^{s_c} \times \dot H^{s_c-1}$ perturbation to the initial data, either outgoing or incoming, without changing the result. In addition, note that the conclusion is still true, but trivial, when $N=4$.

\begin{observation} In the defocusing case all solutions can be conjectured to be dispersive, as suggested by the Morawetz estimate
$$
\int_{\R^3 \times I} \frac {|u|^{N+2}}{|x|} \dd x \dd t \les E[u].
$$
Here $I$ is the maximal interval of existence of the solution $u$. Then condition (\ref{cond}) can be conjectured to always be satisfied (up to a small error) if we wait for long enough. Indeed, the $\dot H^1 \times L^2$ energy norm remains bounded by the energy $E[u]$, so the left-hand side of (\ref{cond}) should improve with time: the solution should become more outgoing and further removed from the origin. Since $4/N-1<0$, $R^{4/N-1} \to 0$ as $R \to \infty$.

Thus, our results could be part of the the process of showing global in time existence and scattering for any large radial solution to (\ref{eq_sup}), after the solution is first shown to disperse for a sufficiently long, but finite time.
\end{observation}


The next result shows that it is not necessary to assume that the initial data have finite energy --- bounded and of compact support will suffice.
\begin{theorem}\lb{bounded_thm} Assume that $N \in [4, \infty)$ and $(u_0, u_1)$ are radial initial data, outgoing according to Definition \ref{def_outgoing}, such that $u_0$ is bounded and supported on $B(0, R)$. Then, as long as
$$
\|u_0\|_{L^\infty} R^{2/N} << 1
$$
is sufficiently small, the corresponding solution $u$ to (\ref{eq_sup}) exists globally, forward in time, and disperses:
$$
\|u\|_{L^{3N/2}_x L^\infty_t \cap L^\infty_x L^{N/2}_t} \les \|u_0\|_{L^\infty} R^{2/N}.
$$
\end{theorem}
The initial data are only in $L^2 \times \dot H^{-1}$ (for which the notion of being outgoing is still well-defined, however), but these solutions have finite homogenous $L^{2N}_{t, x}$ norm, meaning that they preserve higher regularity.

Note that these initial data must still be small in the critical $L^{p_c}$ norm, but can be arbitrarily large in the critical Sobolev norm.

Again, it is possible to add a small $\dot H^{s_c} \times \dot H^{s_c-1}$ perturbation, either incoming or outgoing, to the initial data.

As a complement to Theorem \ref{bounded_thm}, by Proposition \ref{local1} solutions to (\ref{eq_sup}) exist locally for large radial outgoing initial data $(u_0, u_1)$ with $u_0 \in L^\infty$.

\subsection{Large solutions}

In the previous section we constructed global solutions for outgoing initial data $u_0$ of large (possibly infinite) critical $\dot H^{s_c}$ norm. However, these solutions are still small in a certain sense, because $\|u_0\|_{L^{p_c}}$ and $\|u\|_{L^{2N}_{t, x}}$ are small.

In an interesting remark, the referee raised the question of exactly what we mean by a ``large'' solution. In this section we set out to answer this question by constructing a solution that is large according to almost every reasonable standard, while also examining, then discarding several alternative methods.

A more obvious way of constructing large solutions is as follows: start with initial data $(u_0, u_1)$ of large, but finite $\dot H^{s_c} \times \dot H^{s_c-1}$ norm and let them evolve under the linear flow, until the remaining Strichartz norm of their future evolution becomes small:
\be\lb{other_outgoing}
\|\Phi_0(t)(u_0, u_1)\|_{L^{2N}(\R^3 \times (T, \infty))} << 1.
\ee
Then $(\tilde u_0, \tilde u_1) := \Phi(T)(u_0, u_1)$ is still large in $\dot H^{s_c} \times \dot H^{s_c-1}$ and they give rise to small global solutions to (\ref{eq_sup}) forward in time, but $\|\tilde u_0\|_{L^{p_c}}<<1$.

This is a particular case of our construction in Theorem \ref{thm_main} and Corollary \ref{cor_sup}, because after a long time any linear solution becomes almost completely outgoing and supported far from the origin. Our construction is more general, since we only assume the initial data are outgoing, not that they are supported far away.


Another way of constructing large solutions is by superposing many small profiles, widely separated in either space or scale (so that the nonlinear interaction between them is minimized). However, as we shall see below, any solution constructed in this manner is still necessarily ``locally'' small, because each bump is small and they are widely separated. Thus, we just need an appropriate norm to take advantage of this smallness.

In the radial setting, as suggested by the anonymous referee, one can construct a large solution as follows: take small radial initial data $(\phi, \psi)$, either in the $\dot H^{s_c}$ norm or as in Theorem \ref{bounded_thm}, so that the corresponding solution $u$ to (\ref{eq_sup}) is global and scattering. For
$$
0 < \lambda_1 << \lambda_2 << \ldots << \lambda_J,
$$
take
\be\lb{different_scales}
u_0 = \sum_{j=1}^J \lambda_j^{-2/N} \phi(x/\lambda_j),\ u_1 = \sum_{j=1}^J \lambda_j^{-1-2/N} \psi(x/\lambda_j).
\ee

Note that these initial data are large in the $L^{p_c}$ norm. Indeed, if the scales $\lambda_j$ are sufficiently separated, then $\|u_0\|_{L^{p_c}}^{p_c} \sim J \|\phi\|_{L^{p_c}}^{p_c}$. However, the initial data $(u_0, u_1)$ are still small in $\dot B^{s_c}_{2, \infty} \times \dot B^{s_c-1}_{2, \infty}$ (where also $\dot B^{s_c}_{2, \infty} \subset L^{p_c, \infty}$).

Thus, for $s_c>1$, we immediately get from \cite{becgol} that
$$
\|u\|_{L^{3N/2, \infty}_x L^\infty_t} \les \|u_0\|_{\dot B^{s_c}_{2, \infty}} + \|u_1\|_{\dot B^{s_c-1}_{2, \infty}} << 1
$$
(this follows by real interpolation, see \cite{bergh}). The equation (\ref{eq_sup}) is well-posed (in particular, higher regularity is preserved) in this case, by the same proof as Theorem \ref{bounded_thm}, so (\ref{different_scales}) are still small data in some very precise sense.

Discarding radial symmetry, one can also construct a global solution with large initial data given by many small bumps far apart from one another:
\be\lb{farapart}
u_0 = \sum_{j=1}^J \phi(x-y_j),\ u_1 = \sum_{j=1}^J \psi(x-y_j),\ |y_{j_1}-y_{j_2}| >> 1\ \forall j_1 \ne j_2.
\ee
Clearly, such data are not small in any scaling- and symmetric rearrangement-invariant norm. However, this just means that we need a different type of norm to see their smallness. We construct one in the Appendix, by means of the Choquet integral (see \cite{choquet} and \cite{adams}) corresponding to an outer norm defined in terms of the global Kato norm (see \cite{goldberg}).

More precisely, we show that one can globally solve equation (\ref{eq_sup}) by means of a contraction argument in $L^{p, \infty}_x(\mu_\alpha) L^t_\infty$, see Theorem \ref{closed_loop}. The linear evolution and the solution spanned by initial data (\ref{farapart}) are always small in this space.

This more general scaling-invariant norm also controls (\ref{different_scales}), as well as any combination of small bumps separated in scale and/or space. Beyond that, this norm can also be used to control a solution that does not decay at spatial infinity, as long as it is sufficiently sparse; we obtain a quantitative estimate of the sparseness required. See the Appendix for details.

Finally, using the methods introduced in this paper, we construct a solution which is large by all the standards described above. For simplicity we assume that $N$ is an even integer.
\begin{theorem}[Main result]\lb{large_data} Assume that $N \in (2, \infty)$ is even. For any $L>0$ there exist radial initial data $(u_0, u_1)$ such that $\|u_0\|_{L^{p_c}} \geq L$, the corresponding solution $u$ of (\ref{eq_sup}) is global, forward in time, and
$$
\|u\|_{\langle x \rangle^{-1} L^\infty_{t, x} \cap \langle x \rangle^{-1} L^\infty_x L^1_t} < \infty,\ \sup_t u(x, t) \gtrsim L \langle x \rangle^{-1}.
$$
\end{theorem}
Our construction is based on taking outgoing initial data concentrated on a thin spherical shell.

Such solutions have finite critical $L^{2N}_{t, x}$ norm, so they are also stable under small $\dot H^{s_c} \times \dot H^{s_c-1}$ perturbations.

\begin{observation}[Size of the initial data and solution] 1. These initial data necessarily have arbitrarily high $\dot H^{s_c}$ norm, but we already had such examples from Theorem \ref{thm_main}.

2. From our construction it follows that $\|u_0\|_{L^\infty(|x| \geq 1)}>>1$ (another standard for largeness), see \cite{krsc} for a similar result obtained by completely different methods.

3. More generally, it is reasonable to consider the seminorms $\|u\|_{L^p(|x| \geq 1)}$, $p>p_c$, and $\|u\|_{L^p(|x| \leq 2)}$, $p<p_c$. By allowing for rescaling and translation, these seminorms become norms:
$$
\|u\|:= \sup_{\lambda, y} \|\lambda^{-2/N} u(x/\lambda - y)\|_{L^p(|x| \leq 1)},\ p < p_c
$$
or
$$
\|u\|:= \sup_{\lambda, y} \|\lambda^{-2/N} u(x/\lambda - y)\|_{L^p(|x| \geq 1)},\ p > p_c.
$$
Our initial data $u_0$ are also large in these norms when $p>N+1$.

4. Finally, the solution $u$ is large in the sense of Theorem \ref{closed_loop}. Indeed, the norms $\|\Phi_0(u_0, u_1)\|_{L^{p, \infty}_x(\mu_\alpha) L^\infty_t}$ and $\|u\|_{L^{p, \infty}_x(\mu_\alpha) L^\infty_t}$ are uniformly large for all $p$ in the specified range $N+1 \leq p \leq 3N/2$.
\end{observation}

As suggested by the referee, one can perhaps combine the construction in Theorem \ref{large_data} with the ones in (\ref{different_scales}) and/or (\ref{farapart}) to obtain an even larger solution. We shall not pursue this idea here.

All these results hold in both the focusing and the defocusing case, regardless of the sign of the nonlinearity in (\ref{eq_sup}).

Also note that if we assume the initial data are smooth then the solution is also smooth. 

\subsection{History of the problem}
The first well-posedness result for large data supercritical problems was obtained by Tao \cite{tao}, for the logarithmically supercritical defocusing wave equation that he introduced
\be\lb{eqlog}
u_{tt} - \Delta u + u^5 \log(2+u^2) = 0,\ u(0) = u_0,\ u_t(0) = u_1.
\ee
\cite{tao} proved global well-posedness and scattering for radial initial data. The starting point of \cite{tao} was an observation made in \cite{gsv} for the energy-critical problem.

Further results belong to Roy \cite{roy} \cite{roy2}, who proved the scattering of solutions to the log-log-supercritical equation
$$
u_{tt} - \Delta u + u^5 \log^c(\log(10+u^2)) = 0,\ u(0) = u_0,\ u_t(0) = u_1,
$$
$0<c<\frac 8 {225}$, without the radial assumption.

Struwe \cite{struwe} proved the global well-posedness of the equation
$$
u_{tt} - \Delta u + u e^{u^2} = 0,\ u(0)=u_0,\ u_t(0)=u_1,
$$
(supercritical when $E[u]>2\pi$) for arbitrary radial smooth initial data.

Another series of results asserts the conditional well-posedness of supercritical equations. If the critical Sobolev norm $\|(u, u_t)\|_{\dot H^{s_c} \times \dot H^{s_c-1}}$ of a solution to (\ref{eq_sup}) stays bounded, then the solution exists globally and disperses. Such findings belong to Kenig--Merle \cite{keme2}, Killip--Visan \cite{kivi1}, \cite{kivi2}, and Bulut \cite{bul1}, \cite{bul2}, \cite{bul3} in the defocusing case and Duyckaerts--Kenig--Merle \cite{dkm}, Dodson--Lawrie \cite{dola}, and Duyckaerts--Roy \cite{duro} in the focusing case.

All these conditional results are based on methods developed by Bourgain \cite{bou}, Colliander--Keel--Staffilani--Takaoka--Tao \cite{ckstt}, Kenig--Merle \cite{keme}, and Keraani \cite{ker} in the energy-critical case.

By an original method, adapted to that specific equation, Li \cite{li} proved the unconditional global wellposedness of hedgehog solutions for the $(3+1)$ Skyrme model.

Wang--Yu \cite{wang}, Yang \cite{shiwu}, and Miao--Pei--Yu \cite{mpy} constructed large global solutions for semilinear wave equations satisfying the null condition, related to a result of Christodulou \cite{christo}.

Finally, another recent result belongs to Krieger--Schlag \cite{krsc}, who construct a very specific class of global, \emph{smooth} nondispersive solutions to (\ref{eq_sup}) with the best decay at infinity possible. Due to the slow decay, these solutions logarithmically miss being in $\dot H^{s_c} \times \dot H^{s_c-1}$, but belong instead to the Besov spaces $\dot B^{s_c}_{2, \infty} \times \dot B^{s_c-1}_{2, \infty}$.

The results in this paper are in the same spirit as \cite{wang}, \cite{shiwu}, \cite{mpy}: taking a particular class of initial data with much better properties than generic ones and constructing large solutions for them.

These previous papers use a null frame decomposition of energy-class solutions (and assume the finiteness of some higher order energy norms). One component of the solution is allowed to be large, all others are assumed to be small, then the null condition for the nonlinearity prevents large-large nonlinear self-interactions.

By contrast, our results are based on a decomposition of possibly much rougher (infinite energy) solutions into incoming and outgoing components. The linear evolution of outgoing initial data has better properties than generic solutions and satisfies improved multilinear estimates. The nonlinearity in (\ref{eq_sup}) does not satisfy the null condition.

For more recent and roughly similar results, also see Luk--Oh--Yang \cite{loy}, who construct large solutions of Einstein's equation and of equation (\ref{eq_sup}) with radial symmetry. Their solutions have infinite critical norm due to slow decay, like those in \cite{krsc}.

Our result and the one of \cite{krsc} are also rather different. Ours is based on multilinear estimates and \cite{krsc} is based on a nonlinear construction. In addition, the solutions of \cite{krsc} only logarithmically fail to be in $\dot H^{s_c} \times \dot H^{s_c-1}$ and in fact are bounded in a critical Besov space. By contrast, our solutions can miss being in the critical Sobolev space by a wide margin and are large in all reasonable critical norms.

In this paper, we do not require solutions to be bounded or even finite in the critical Sobolev norm $\dot H^{s_c} \times \dot H^{s_c-1}$. Indeed, this critical norm can be replaced with the $\dot H^1 \times L^2$ energy norm of outgoing initial data, provided they are supported sufficiently far out (or the $L^\infty$ norm is sufficiently small).


One can use a numerical scheme or some other approximation procedure, which is accurate over short times, for specific initial conditions. If the solution, after such a short time, satisfies our outgoing conditions up to a small error, then it will be global by Corollary \ref{cor_sup}.

This is plausible because it is true in the free case: any solution to the free wave equation becomes outgoing and supported far away from the origin, up to a small error, when a sufficiently long time has elapsed.

Energy being finite is also not necessary, as we construct solutions for bounded and compactly supported initial data. A variant of our construction, Theorem \ref{large_data}, leads to initial data that are large in every possible sense, see the discussion above.

Even though our results hold for rough initial data, we can also assume both the initial data and the solutions are smooth, due to the preservation of regularity.

In other contexts, various notions of incoming and outgoing waves have been introduced and used. However, the incoming and outgoing projections that we define seem to be new. In papers on this topic, ``outgoing'' typically refers to any linear solution after a sufficient time has elapsed so that the remaining Strichartz norm of its future evolution is small, see (\ref{other_outgoing}). We need no such smallness assumption.

The incoming condition in this paper resembles a condition from Engquist--Majda \cite{enma}, see formula (1.27) in that paper.

We expect the same method to lead to an improvement in the energy-critical and subcritical cases, by allowing us to prove, for example, global well-posedness for $(\dot H^{1/2} \cap L^\infty) \times \dot H^{-1/2}$ outgoing initial data, thus requiring fewer derivatives than the critical Sobolev exponent for the equation. 
These improvements will be explored in a future paper.

Another expected result is the well-posedness of equation (\ref{eq_sup}) for arbitrary large initial data, after projecting the nonlinearity on the outgoing states. This constitutes the subject of our next paper, \cite{becsof1}.


This paper is organized as follows: in Section \ref{outgoing_incoming} we state several results about incoming and outgoing states for the linear flow, in Section \ref{well_posedness} we enounce some standard existence results, in Section \ref{proof_results} we prove the theorems stated in the introduction, and in the Appendix we introduce Lorentz--Choquet norms for the Kato outer measure and use them to study multi-bump solutions.

In the initial version of this paper there was one more result in the Appendix, concerning large solutions for the focusing supercritical wave equation, obtained by means of a positivity criterion. We have removed that result and developed it into a separate paper \cite{becsof2}.

\section{Notations}
$A \les B$ means that $|A| \leq C |B|$ for some constant $C$. We denote various constants, not always the same, by $C$.

The Laplacian is the operator on $\R^3$ $\Delta=\frac {\partial^2}{\partial_{x_1}^2} + \frac {\partial^2}{\partial_{x_2}^2} + \frac {\partial^2}{\partial_{x_3}^2}$.

We denote by $L^p$ the Lebesgue spaces, by $\dot H^s$ and $\dot W^{s, p}$ (fractional) homogenous Sobolev spaces, and by $L^{p, q}$ Lorentz spaces. We also define the weighted Lebesgue spaces $w(x) L^p_x := \{w(x) f(x): f \in L^p\}$.

$\dot H^s$ are Hilbert spaces and so is $\dot H^1 \times L^2$, under the norm
$$
\|(u_0, u_1)\|_{\dot H^1 \times L^2}=(\|u_0\|_{\dot H^1}^2+\|u_1\|_{L^2}^2)^{1/2}.
$$

By $\dot H^s_{rad}$, etc.\ we designate the radial version of these spaces. For a radially symmetric function $u(x)$, we let $u(r):=u(x)$ for $|x|=r$.

By $(\dot H^1 \times L^2)_{out}$ we mean the space of outgoing radially symmetric $\dot H^1 \times L^2$ initial data, see Definition \ref{def_1}.

We define the mixed-norm spaces on $\R^3 \times [0, \infty)$
$$
L^p_t L^q_x := \Big\{f \mid \|f\|_{L^p_t L^q_x}:= \Big(\int_0^\infty \|f(x, t)\|_{L^q_x}^p \dd t\Big)^{1/p} < \infty \Big\},
$$
with the standard modification for $p=\infty$, and likewise for the reversed mixed-norm spaces $L^q_x L^p_t$. We use a similar definition for $L^p_t \dot W^{s, p}_x$. Also, for $I \subset [0, \infty)$, let $\|f\|_{L^p_t L^q_x(\R^3 \times I)} := \|\chi_I(t) f\|_{L^p_t L^q_x}$, where $\chi_I$ is the characteristic function of $I$.

We also denote $B(0, R):=\{x \in \R^3 \mid |x| \leq R\}$.

Let $D$ be the Fourier multiplier $|\xi|$, $\delta_0$ be Dirac's delta at zero, and $\chi$ denote the indicator function of a set.

Let $\Phi(t):\dot H^1 \times L^2 \to \dot H^1 \times L^2$ be the flow of the linear wave equation in three dimensions: for
$$
u_{tt}-\Delta u=0,\ u(0)=u_0,\ u_t(0)=u_1,
$$
we set $\Phi(t)(u_0, u_1)=(\Phi_{0}(t)(u_0, u_1), \Phi_{1}(t)(u_0, u_1)):=(u(t), u_t(t))$.

Also let $\phi(t):L^2 \times \dot H^{-1} \to L^2 \times \dot H^{-1}$ be the flow of the linear wave equation in dimension one (on a half-line with Neumann boundary conditions): for
$$
v_{tt}-v_{rr}=0,\ v(0)=v_0,\ v_t(0)=v_1, v_r(0, t)=0,
$$
we set $\phi(t)(v_0, v_1):=(v(t), v_t(t))$.

In this paper we only consider mild solutions to (\ref{eq_sup}), i.e.\ solutions to the following equivalent integral equation:
$$
u(t) = \cos(t\sqrt{-\Delta})u_0 + \frac {\sin(t\sqrt{-\Delta})}{\sqrt{-\Delta}} u_1 \mp \int_0^t \frac {\sin((t-s)\sqrt{-\Delta})}{\sqrt{-\Delta}} |u(s)|^Nu(s) \dd s.
$$

\section{Outgoing and incoming states for the free flow}\lb{outgoing_incoming}
In order to define outgoing and incoming states for the linear wave equation, we reduce the equation to a one-dimensional problem, where identifying such states is straightforward.

In the radial case, the following operator will establish a correspondence between the three-dimensional wave equation and the one-dimensional wave equation on a half-line:
\begin{definition}
\be\lb{tu}
T(u)(r) = \frac 1 {2\pi} (ru(r))',\ u(r) = \frac {2\pi} r \int_0^r T(u)(s) \dd s.
\ee
\end{definition}


\begin{observation}
If we write the radial function $u(r=|x|)$ on $\R^3$ as a superposition of identical one-coordinate functions in every possible direction, each of these functions can be taken to be $T(u)$. In other words, for each radial function $u(x)$ on $\R^3$
\be\lb{radial}
u(x) = \int_{S^2} T(u)(x \cdot \omega) \dd \omega.
\ee

For $T(u)$ supported on $[0, \infty)$, this works out to (\ref{tu}). Indeed, with no loss of generality assume that $x=(0, 0, r)$ and write $\omega$ in polar coordinates as $\omega=(\sin \theta \cos \phi, \sin \theta \sin \phi, \cos \theta)$. Then
$$
u(r)= \int_0^\pi \int_0^{2\pi} T(u)(r\cos \theta) \sin \theta \dd \phi \dd \theta = \frac {2\pi} r \int_{-r}^r T(u)(s) \dd s = \frac {2\pi} r \int_0^r T(u)(s) \dd s.
$$

One can further generalize this analysis to the non-radial case, by using the Radon transform for the construction of the projections.
\end{observation}

By $\dot H^{-1}([0, \infty))$ we understand the space of distributions that are derivatives of $L^2$ functions and are supported on $[0, \infty)$.

Conveniently, $T$ is a constant times a unitary map from $L^2_{rad}$ (the space of radial $L^2$ functions) to $\dot H^{-1}([0, \infty))$ and bounded from $\dot H^1_{rad}$ to $L^2([0, \infty))$ --- the latter by Hardy's inequality. The inverse operator $T^{-1}$ is also bounded from $L^2([0, \infty))$ to $\dot H^1_{rad}$.
\begin{lemma}\lb{equivalent}
With $T$ defined by (\ref{tu}), $\|T(u)\|_{\dot H^{-1}([0, \infty))} = \frac 1 {\sqrt \pi} \|u\|_{L^2_{rad}}$. Moreover, $\|T(u)\|_{L^2([0, \infty))} \les \|u\|_{\dot H^1_{rad}}$ and $\|u\|_{\dot H^1_{rad}} \les \|T(u)\|_{L^2([0, \infty))}$.
\end{lemma}
\begin{observation}\lb{norma} This shows that $\|u\|:=\|T(u)\|_{L^2([0, \infty))}=\|(ru(r))'\|_{L^2([0, \infty))}$ is another norm on $\dot H^1_{rad}$ equivalent to the usual one.
\end{observation}
\begin{proof}[Proof of Lemma \ref{equivalent}] Note that $u \in L^2_{rad}$ if and only if
$$
\|u\|_{L^2_{rad}}^2 = 4\pi \int_0^\infty |u(r)|^2 r^2 \dd r.
$$
In particular, $\|u\|_{L^2_{rad}}=2\sqrt\pi\|ru(r)\|_{L^2([0, \infty))}=2\sqrt\pi\|(ru(r))'\|_{\dot H^{-1}([0, \infty))}$.

Next, note that by Hardy's inequality
$$
\|u/r\|_{L^2_{rad}}^2 = 4\pi \int_0^\infty |u(r)|^2 \dd r \les \|u\|_{\dot H^1_{rad}}^2 = 4\pi \int_0^\infty |u_r(r)|^2 r^2 \dd r.
$$
Therefore $\|(ru(r))'\|_{L^2([0, \infty))} \leq \|u\|_{L^2_{[0, \infty)}}+\|ru'(r)\|_{L^2([0, \infty))} \les \|u\|_{\dot H^1_{rad}}$.

Finally, let $T(u)=v$. Then by (\ref{tu})
$$
\|u\|_{\dot H^1_{rad}}^2 = 4\pi \int_0^\infty |u_r(r)|^2 r^2 \dd r \les \int_0^\infty |v(r)|^2 \dd r + \int_0^\infty \frac 1 {r^2} \Big(\int_0^r v(s) \dd s\Big)^2 \dd r.  
$$
Then, $\frac 1 r \int_0^r v(s) \dd s = \int_0^1 v(\alpha r) \dd \alpha$ and $\|v(\alpha \cdot)\|_{L^2} = \alpha^{-1/2} \|v\|_{L^2}$, so
$$
\Big\|\int_0^1 v(\alpha r) \dd \alpha\Big\|_{L^2_r} \leq \int_0^1 \|v(\alpha \cdot)\|_{L^2} \dd \alpha = \|v\|_{L^2} \int_0^1 \alpha^{-1/2} \dd \alpha \les \|v\|_{L^2}.
$$
This proves the last statement of the lemma.
\end{proof}

For non-radial functions, a similar decomposition into one-coordinate functions can be obtained, but the computation is more complicated. One restricts the three-dimensional Fourier transform along each line through the origin, then takes the inverse one-dimensional Fourier transform. This is related to the Radon transform.

So far, the transformation $T$ is completely general;  however, the subsequent computation is not and will be different for, say, Schr\"{o}dinger's equation.

\begin{lemma} There exist bounded operators $P_+$ and $P_-$ on $\dot H^1_{rad} \times L^2_{rad}$, given by
\be\lb{outgoing}
P_+(u_0, u_1) = \Big(\frac 1 2 \Big(u_0 - \frac 1 r \int_0^r s u_1(s) \dd s\Big), \frac 1 2 \big(-(u_0)_r-\frac {u_0} r + u_1\big)\Big)
\ee
and
\be\lb{incoming}
P_-(u_0, u_1) = \Big(\frac 1 2 \Big(u_0 + \frac 1 r \int_0^r s u_1(s) \dd s\Big), \frac 1 2 \big((u_0)_r+\frac {u_0} r + u_1\big)\Big),
\ee
such that $I=P_++P_-$, $P_+^2=P_+$, and $P_-^2=P_-$.

If $\Phi(t)$ is the flow of the linear equation then for $t \geq 0$ $P_- \Phi(t) P_+=0$ and for $t \leq 0$ $P_+ \Phi(t) P_-=0$. In addition, for $t \geq 0$ $\Phi(t)P_+(u_0, u_1)$ is supported on $\ov{\R^3\setminus B(0, t)}$ and for $t \leq 0$ $\Phi(t)P_-(u_0, u_1)$ is supported on $\ov{\R^3 \setminus B(0, -t)}$.
\end{lemma}
\begin{definition}\lb{def_1} $P_+$ and $P_-$ are called the projection on outgoing, respectively incoming states. We call any radial $(u_0, u_1)$ such that $P_-(u_0, u_1)=0$ \emph{outgoing}; if $P_+(u_0, u_1)=0$ we call it \emph{incoming}.
\end{definition}
\begin{observation} $P_+$ and $P_-$ are self-adjoint on $\dot H^1_{rad} \times L^2_{rad}$ with the norm $\|(u_0, u_1)\|:=(\|(ru_0(r))'\|_{L^2([0, \infty))}^2+\|ru_1(r)\|_{L^2([0, \infty))}^2)^{1/2}$, see Remark \ref{norma}. 
\end{observation}

\begin{proof} Given a radial solution $u$ of the free wave equation
\be\lb{free_wave}
u_{tt} - \Delta u = 0,\ u(0)=u_0,\ u_t(0)=u_1,
\ee
each of the essentially one-dimensional functions $T(u(t))(x\cdot\omega)$ fulfills a one-dimensional wave equation of the form
\be\lb{eq_1}
v_{tt} - v_{rr} = 0,\ v(0) = v_0,\ v_t(0) = v_1,
\ee
where
$$
v_0=\frac 1 {2\pi}(ru_0(r))',\ v_1=\frac 1 {2\pi}(ru_1(r))'.
$$
Indeed, in radial coordinates one has that $u_{tt} - u_{rr} - (2/r) u_r = 0$ or equivalently that $(ru)_{tt}-(ru)_{rr}=0$. Since $v=T(u)=\frac 1 {2\pi} (ru(r))'$, taking a derivative leads to (\ref{eq_1}), which holds in the weak sense.

To fix ideas we assume that $(u_0, u_1) \in \dot H^1_{rad} \times L^2_{rad}$, so $(v_0, v_1) \in L^2 \times \dot H^{-1}$ by Lemma \ref{equivalent}. Note that $v_0$ and $v_1$ are supported on $[0, \infty)$.

In addition, 
we consider equation (\ref{eq_1}) on a half-line only and impose the Neumann boundary condition $v_r(0, t)=0$. This is justified because for a smooth radial solution $u$ one necessarily has $u_r(0, t)=0$ and $v_r(0, t)=\frac 1 {2\pi}(ru)_{rr}(0, t)=\frac 1 \pi u_r(0, t)$.

The Neumann boundary condition means that the solution is reflected back at the boundary or, in other words, that it could be extended by symmetry to the negative half-axis.

In fact, note that we can prove some of the statements in just the case when $u_0$ and $u_1$ are smooth functions and proceed by continuity in the general case, since $P_+$ and $P_-$ are bounded on $\dot H^s \times \dot H^{s-1}$, $1 \leq s<3/2$, as we prove below.

Equation (\ref{eq_1}) then has solutions of the form (with $r \geq 0$)
$$
v(r, t) = \chi_{r \geq t} v_+(r-t) + \chi_{r \leq t} v_-(t-r) + \chi_{r+t\geq 0} v_-(r+t) + \chi_{r+t\leq 0} v_+(-r-t),
$$
where by d'Alembert's formula
\be\lb{dalembert}
v_+(r)=\frac 1 2 \big(v_0(r)-\partial_r^{-1} v_1(r)\big),\ v_-(r)=\frac 1 2 \big(v_0(r)+\partial_r^{-1} v_1(r)\big).
\ee
Here $\partial_r^{-1}$ denotes the unique antiderivative of a $\dot H^{-1}$ distribution that belongs to $L^2$. Note that since $v_1$ is supported on $[0, \infty)$, $\partial_r^{-1} v_1$ is also supported on $[0, \infty)$ (in fact it is given by $\frac 1 {2\pi} r u_1(r)$).

Thus both $v_+$ and $v_-$ are supported on $[0, \infty)$.

At time $t$, the outgoing component of $v$, which moves in the positive direction with velocity $1$, consists of
$$
v_{out}(r, t)= \chi_{r \geq \max(t, 0)} v_+(r-t) + \chi_{0 \leq r \leq t} v_-(t-r).
$$
The incoming component consists of
$$
v_{in}(r, t)= \chi_{r \geq \max(0, -t)} v_-(r+t) + \chi_{0 \leq r \leq -t} v_+(-r-t).
$$
In particular, at $t=0$ the outgoing component is $v_+$ and the incoming component is $v_-$.

Note that as $t$ grows the incoming component hits the origin and becomes outgoing. If we wait for long enough, most of the solution becomes outgoing. Conversely, if we reverse time flow, as $t \to -\infty$ all of the solution becomes incoming.

In order to obtain a general formula for the outgoing projection, without loss of generality we restrict our attention to time $0$. Let $\pi_+(v_0, v_1)$ be the initial data of the outgoing component, i.e.
$$
\pi_+(v_0, v_1)(r):=(v_+(r-t) \mid_{t=0}, \partial_t (v_+(r-t)) \mid_{t=0})=(v_+(r), -v_+'(r))
$$
and likewise
$$
\pi_-(v_0, v_1)(r):=(v_-(r+t) \mid_{t=0}, \partial_t (v-(r+t)) \mid_{t=0})=(v_-(r), v_-'(r)).
$$
Note that $\pi_+(v_0, v_1)$ is the initial data for the solution $v_+(r-t)$ of equation (\ref{eq_1}) on the time interval $[0, \infty)$ (which moves with velocity $1$ in the positive direction) and same for $\pi_0$ on $(-\infty, 0]$.

By d'Alembert's formula (\ref{dalembert}), $\pi_+$ then has the form
$$
\pi_+(v_0, v_1)(r) := \Big(\frac 1 2\big(v_0(r)-\partial_r^{-1}v_1(r)\big),\ \frac 1 2 \big(-v_0'(r)+v_1(r)\big)\Big)
$$
and the incoming component $\pi_-$ has the form
$$
\pi_-(v_0, v_1)(r) := \Big(\frac 1 2\big(v_0(r)+\partial_r^{-1} v_1(r)\big),\ \frac 1 2 \big(v_0'(r)+v_1(r)\big)\Big).
$$

Note that $\pi_++\pi_-=I$, $\pi_+^2=\pi_+$, $\pi_-^2=\pi_-$, and $\pi_+\pi_-=\pi_-\pi_+=0$. In particular, note that $\partial_r^{-1} v_0' = v_0$, because in any case an antiderivative of $v_0'$ must be of the form $v_0+c$ and this is in $L^2$ only if $c=0$.

Since by definition $\langle u, v\rangle_{L^2}=\langle u', v'\rangle_{\dot H^{-1}}$, a simple computation shows that $\pi_+$ and $\pi_-$ are bounded, self-adjoint operators on $L^2 \times \dot H^{-1}$.

Thus, both $\pi_+$ and $\pi_-$ are orthogonal projections, in a proper setting (on $L^2 \times \dot H^{-1}$ or more generally on $\dot H^s \times \dot H^{s-1}$).

Note that a solution of the form $v:=v_+(r-t)$ preserves the property that $\partial_t v = -\partial_r v$ and hence that $\pi_-(v(t),v_t(t))=0$ for all $t \geq 0$. In other words, if we denote by $\phi(t)$ the flow induced by the linear equation (\ref{eq_1}) on $L^2 \times \dot H^{-1}$, then for $t \geq 0$
$$
\pi_- \phi(t) \pi_+(v_0, v_1) = 0.
$$
Furthermore, in this case $\phi(t) \pi_+(v_0, v_1)$ is obviously supported on $[t, \infty)$.

Likewise, for $t \leq 0$ $\pi_+ \phi(t) \pi_-(v_0, v_1) = 0$ and $\supp \phi(t) \pi_-(v_0, v_1) \subset~[-t, \infty)$.

The outgoing component of $u$ corresponds to the outgoing component $\pi_+(v, v_t)$ of $v$ traveling in the positive direction. Conjugating the projections $\pi_+$ and $\pi_-$ by the transformation $T$ defined by (\ref{radial}), we obtain the corresponding operators for radial functions in $\R^3$. Letting $P_+:=T^{-1}\pi_+T$, $P_-:=T^{-1}\pi_-T$, we obtain formulas (\ref{outgoing}) and (\ref{incoming}). 
Both operators are bounded on $\dot H^1_{rad} \times L^2_{rad}$ due to Lemma \ref{equivalent}.

As an easy consequence of the properties of $\pi_+$ and $\pi_-$, we get that $P_++P_-=I$, $P_+ P_- = 0$, $P_+^2 = P_+$, $P_-^2 = P_-$, and all the other stated properties of $P_+$ and $P_-$.
\end{proof}

We now prove some properties of the nonlocal operator that appears in the definition of the projections on incoming and outgoing states. This leads among other things to the boundedness of the projections on outgoing and incoming states. Also note that if the initial data $(u_0, u_1)$ are purely outgoing or are purely incoming, then $u_0$ determines $u_1$ and vice-versa. Furthermore, this can be made into a quantitative estimate.
\begin{lemma}\lb{equivalence} For radial $f \in L^2$
$$
\Big\|\frac 1 r \int_0^r \rho f(\rho) \dd \rho\Big\|_{\dot H^1_{rad}} \les \|f\|_{L^2_{rad}}.
$$
More generally, for $0 \leq s < 3/2$
\be\lb{embedding}
\Big\|\frac 1 r \int_0^r \rho f(\rho) \dd \rho\Big\|_{\dot H^{s+1}_{rad}} \les \|f\|_{\dot H^s_{rad}}.
\ee
Consequently, $P_+$ and $P_-$ are bounded on $\dot H^s \times \dot H^{s-1}$ for $1 \leq s < 3/2$.

Furthermore, if $(u_0, u_1)$ are purely outgoing or purely incoming, then $\|u_0\|_{\dot H^s} \sim \|u_1\|_{\dot H^{s-1}}$ for $1 \leq s < 3/2$.
\end{lemma}
\begin{proof}[Proof of Lemma \ref{equivalence}]
By differentiation we see that
$$
\Big\|\frac 1 r \int_0^r \rho f(\rho) \dd \rho\Big\|_{\dot H^1_{rad}} = \Big\|f - \frac 1 {r^2} \int_0^r \rho f(\rho) \dd \rho\Big\|_{L^2_{rad}}
$$
and then
\be\lb{unu}\begin{aligned}
&\Big\|\frac 1 {r^2} \int_0^r \rho f(\rho) \dd \rho\Big\|_{L^2_{rad}} \les \Big\|\frac 1 r \int_0^r \rho f(\rho) \dd \rho\Big\|_{L^2_r([0, \infty))} = \Big\|\int_0^1 \alpha r f(\alpha r) \dd \alpha\Big\|_{L^2_r([0, \infty))} \\
&\les \int_0^1 \|\alpha r f(\alpha r)\|_{L^2_r([0, \infty))} = \|rf(r)\|_{L^2_r([0, \infty))} \int_0^1 \alpha^{-1/2} \dd \alpha \les \|f\|_{L^2_{rad}}.
\end{aligned}\ee
	
By the same reasoning, (\ref{embedding}) follows for $0 \leq s \leq 1$ from
$$
\Big\|\frac 1 {r^2} \int_0^r \rho f(\rho) \dd \rho\Big\|_{\dot H^s} \les \|f\|_{\dot H^s}.
$$
This in turn follows by interpolation between the $s=0$ case proved above (see \ref{unu}) and the $s=1$ case, which is implied by Hardy's inequality $\|f/|x|\|_{L^2} \les \|f\|_{\dot H^1}$ and
$$\begin{aligned}
\Big\|\frac 1 {r^3} \int_0^r \rho f(\rho) \dd \rho\Big\|_{L^2} &\leq \Big\|\frac 1 {r^2} \int_0^r f(\rho) \dd \rho\Big\|_{L^2} \les \Big\|\frac 1 r \int_0^r f(\rho) \dd \rho\Big\|_{L^{6, 2}} \\
&= \Big\|\int_0^1 f(\alpha r) \dd \alpha\Big\|_{L^{6, 2}} \les \|f\|_{\dot H^1} \int_0^1 \alpha^{-1/2} \dd \alpha \les \|f\|_{\dot H^1}.
\end{aligned}$$
In the same manner one proves that for $1 \leq s < 3/2$
$$
\Big\|\frac 1 {r^4} \int_0^r \rho f(\rho) \dd \rho\Big\|_{\dot H^{s-2}} \les \Big\|\frac 1 r \int_0^r f(\rho) \dd \rho\Big\|_{\dot H^s} \les \|f\|_{\dot H^s},
$$
which implies that (\ref{embedding}) is true for $0 \leq s < 3/2$.

The $\dot H^s \times \dot H^{s-1}$ boundedness of $P_+$ and $P_-$ for $1 \leq s < 3/2$ is a consequence of their definition, of (\ref{embedding}), and of Hardy's inequality $\|f/|x|\|_{\dot H^{s-1}} \les \|f\|_{\dot H^s}$. The same is true for the final conclusion.
\end{proof}

Next, we state the most important (and somewhat trivial) identity for outgoing solutions.
\begin{proposition}\lb{formula} If $u$ is a radial solution to the free wave equation (\ref{free_wave}) with outgoing initial data $(u_0, u_1)$, then for $r \geq t \geq 0$
$$
u(r, t) = \frac{r-t}r u_0(r-t)
$$
and $u(r, t) \equiv 0$ for $0 \leq r \leq t$.
\end{proposition}
Note that, as $t$ increases, an outgoing solution keeps constant sign. The computation is different, so this is not true, for negative $t$. Equivalently, the incoming component need not keep a constant sign for positive $t$.

Also note that by a direct computation one can check the outgoing property $P_- (u(t), u_t(t))=0$, i.e.\ $u_t+u_r+u/r=0$.
\begin{proof} This follows from the one-dimensional reduction. Indeed, let $v=T(u)$, where $T$ is given by (\ref{radial}). Since $u$ is outgoing, by definition $v$ is also outgoing, i.e.\ $v(r, t)=v_+(r-t)$ for all $t \geq 0$ and some $v_+$ supported on $[0, \infty)$. Then by (\ref{tu})
$$
r u(r, t) = \int_0^r v(s, t) \dd s = \int_0^r v_+(s-t) \dd s = \int_0^{r-t} v_+(s) \dd s,
$$
which only depends on $r-t$. Therefore $ru(r, t)=(r-t)u(r-t, 0)$. The second conclusion follows because the integral is zero when $r \leq t$.
\end{proof}

This identity immediately leads to improved Strichartz and decay estimates for outgoing solutions.

\begin{corollary}[Uniform bounds]\lb{cor_crit} If $u$ is a radial solution to the free wave equation (\ref{free_wave}) in three dimensions with outgoing initial data $(u_0, u_1)$, then if $u_0 \in L^\infty$
\be\lb{sup}
\|u\|_{L^\infty_{t, x}} \leq \|u_0\|_{L^\infty}
\ee
and if $u_0 \in L^p$ then $\|u\|_{L^\infty_t L^p_x} \leq \|u_0\|_{L^p}$ for $2 \leq p \leq \infty$. In particular, the $L^2$ norm $\|u(x, t)\|_{L^2_x}$ is constant with respect to time for $t \geq 0$. Furthermore, $\|u\|_{L^\infty_t \dot W^{1, p}_x} \les \|u_0\|_{\dot W^{1, p}}$ for $2 \leq p < 3$.
\end{corollary}

Note that the last estimate is better than one would expect from scaling.

\begin{proof}
Inequality (\ref{sup}) follows directly from Proposition \ref{formula}, since for $0\leq t \leq r$ $0 \leq \frac {r-t} r \leq 1$. Concerning the $L^p$ norm, for $2 \leq p<\infty$
$$\begin{aligned}
\|u(x, t)\|_{L^p_x}^p &= 4\pi \int_0^\infty |u(r, t)|^p r^2 \dd r = 4\pi \int_t^\infty \Big(\frac {r-t} r\Big)^p |u_0(r-t)|^p r^2 \dd r \\
&\leq 4\pi \int_0^\infty |u_0(r)|^p r^2 \dd r
\end{aligned}$$
since $\big(\frac {r-t} r\big)^p r^2 \leq (r-t)^2$ for $p \geq 2$.

Note that by dominated convergence, when $2<p<\infty$, in fact $\|u(x, t)\|_{L^p_x} \to 0$ as $t \to \infty$. When $p=\infty$ $\|u(x, t)\|_{L^\infty_x} \to 0$ as $t \to \infty$ for $u_0 \in L^\infty_0$, the closure in $L^\infty$ of the set of bounded functions with compact support.

Regarding the Sobolev norms, again by Proposition \ref{formula}
$$
\|u(x, t)\|_{\dot W^{1, p}}^p = 4\pi \int_0^\infty |u_r(r, t)|^p r^2 \dd r = 4\pi \int_t^\infty \Big|\frac {r-t}r (u_0)_r(r-t) + \frac t {r^2} u_0(r-t)\Big|^p r^2 \dd r.
$$
We bound the first term exactly as above and for the second term we use Hardy's inequality, i.e.
$$
\int_t^\infty \Big(\frac t {r^2}\Big)^p |u_0(r-t)|^p r^2 \dd r \leq \int_0^\infty \Big(\frac {|u_0(r)|}{r}\Big)^p r^2 \dd r \les \|u_0\|_{\dot W^{1, p}}^p
$$
since $t \leq r$ and $p \geq 2$.
\end{proof}

\begin{corollary}[Decay and reversed Strichartz estimates] Suppose that the initial data $(u_0, u_1)$ is outgoing and $\supp u_0 \subset B(0, R)$. Then for $t \geq 0$ and $2 \leq p \leq \infty$, the solution to the free wave equation (\ref{eq_sup}) satisfies
\be\lb{decay}
\|u(x, t)\|_{L^p_x} \les \min(1, \Big(\frac R t\Big)^{-1+2/p}) \|u_0\|_{L^p}.
\ee
Suppose that $\supp u_0 \subset B(0, R)$. Then for $3<q<\infty$ (and $L^{3, \infty}$ or $L^\infty$ at the endpoints) and $1 \leq p \leq \infty$
\be\lb{reversed_est}
\|u\|_{L^{q, 2}_x L^p_t} \les R^{3/q+1/p} \|u_0\|_{L^\infty}
\ee
and for $1 \leq p \leq 2$
$$
\|u\|_{L^{3, \infty}_x L^p_t} \les R^{1/p-1/2} \|u_0\|_{L^2}.
$$
More generally, suppose that $\supp u_0 \subset B(0, R_1) \setminus B(0, R_2)$ for $R_1 > R_2$ and $u_0 \in L^\infty$. Then for $3<q< \infty$ (and $L^{3, \infty}$ or $L^\infty$ at the endpoints) and $1 \leq p \leq \infty$
\be\lb{reversed_est'}
\|u\|_{L^{q, 2}_x L^p_t} \les R_1^{3/q} (R_1-R_2)^{1/p} \|u_0\|_{L^\infty}.
\ee
Also, for $1 \leq p \leq 2$ and $3 < q \leq \infty$ (and $L^{3, \infty}$ for $q=3$)
$$
\|u\|_{L^{q, 2}_x L^p_t} \les (R_1-R_2)^{1/p-1/2} R_2^{-1+3/q} \|u_0\|_{L^2}.
$$
\end{corollary}

\begin{proof}
Estimate (\ref{decay}) is an obvious consequence of Proposition \ref{formula} when $p=\infty$ and we interpolate with $p=2$ (for which $\|u(x, t)\|_{L^2_x} = \|u_0\|_{L^2}$) to get all the other cases.

Next, (\ref{reversed_est}) follows because, when $\supp u_0 \subset B(0, R)$ and $u_0 \in L^\infty$, by Proposition \ref{formula} (where we use the fact that $\frac {r-t} r \leq 1$ on one hand and that $r-t\leq R$ on $\supp u_0$ on the other hand)
$$
\|u(r, t)\|_{L^p_t} \les \min(R^{1/p} \|u_0\|_{L^\infty}, \frac {R^{1+1/p}}{r} \|u_0\|_{L^\infty}).
$$
Likewise, for $1 \leq p \leq 2$
$$
\|u(r, t)\|_{L^p_t} \les \frac {R^{1/p-1/2}}{r} \|u_0\|_{L^2}.
$$

Finally, (\ref{reversed_est'}) is true because, when $\supp u_0 \subset B(0, R_1) \setminus B(0, R_2)$ and $u_0 \in L^\infty$,
$$
\|u(r, t)\|_{L^p_t} \les \min((R_1-R_2)^{1/p} \|u_0\|_{L^\infty}, \frac {R_1(R_1-R_2)^{1/p}}{r} \|u_0\|_{L^\infty}).
$$
Also, for the last inequality, by Proposition \ref{formula}
$$
\|u(r, t)\|_{L^p_t} \les \min\Big(\frac {(R_1-R_2)^{1/p-1/2}} r \|u_0\|_{L^2}, \frac {(R_1-R_2)^{1/p-1/2}} {R_2} \|u_0\|_{L^2}\Big).
$$
\end{proof}

We next state some Strichartz estimates that hold only for outgoing solutions. For simplicity, we state them only for the scaling-invariant norms of our problem (\ref{eq_sup}).
\begin{corollary}[Strichartz estimates] For any $4\leq N\leq \infty$, if $u_0 \in L^\infty$ and $(u_0, u_1) \in (\dot H^1 \times L^2)_{out}$ are radial and outgoing, then the corresponding solution $u$ to the free wave equation (\ref{free_wave}) in three dimensions fulfills
\be\lb{crit}
\|u\|_{L^N_t \dot W^{2/N,\,N}_x} + \|u\|_{L^{N/2}_t L^\infty_x} + \|u\|_{L^\infty_t \dot W^{4/N,\, N/2}_x} \les \|u_0\|_{\dot H^1}^{4/N} \|u_0\|_{L^\infty}^{1-4/N}
\ee
and
\be\lb{crit_2}
\||u|^N u\|_{L^1_t \dot H^{s_c-1}_x} \les \|u_0\|_{\dot H^1}^{(N+1)4/N} \|u_0\|_{L^\infty_{x, t}}^{(N+1)(1-4/N)}.
\ee
\end{corollary}
Note that the bounds (\ref{crit}) hold for less than the full range of scaling-invariant norms.

\begin{proof}
Strichartz estimates for the free wave equation (see \cite{give} or \cite{keeltao}, as well as \cite{klma} for the radial endpoint estimate) ensure that
$$
\|u\|_{L^4_t  \dot W^{1/2,\,4}_x} + \|u\|_{L^5_t L^{10}_x} + \|u\|_{L^2_t L^\infty_x} + \|u\|_{L^\infty_t \dot H^1_x} \les \|u_0\|_{\dot H^1} + \|u_1\|_{L^2} \les \|u_0\|_{\dot H^1},
$$
where we also used Lemma \ref{equivalence}. Interpolating (see Theorems 5.1.2 and 6.4.5 in \cite{bergh} for the interpolation results) with the supremum estimate (\ref{sup}), we obtain that for $N \geq 4$
$$
\|u\|_{L^N_t \dot W^{2/N, N}_x} + \|u\|_{L^{N/2}_t L^\infty_x} + \|u\|_{L^\infty_t \dot W^{4/N, N/2}_x} \les \|u_0\|_{\dot H^1}^{4/N} \|u_0\|_{L^\infty}^{1-4/N},
$$
which is the scaling-invariant estimate (\ref{crit}).

By the Leibniz rule (only here we use that $N$ is an integer and it is probably unnecessary), for any integer $N \geq 4$
\be\lb{eq1}\begin{aligned}
\||u|^N u\|_{L^1_t \dot H^1_x} &\les \|u\|_{L^\infty_t \dot H^1_x} \|u\|_{L^2_t L^\infty_x}^2 \|u\|_{L^\infty_{t, x}}^{N-2}  \les \|u_0\|_{\dot H^1}^3 \|u_0\|_{L^\infty}^{N-2}
\end{aligned}\ee
and by H\"{o}lder's inequality
\be\lb{eq2}
\||u|^Nu\|_{L^1_t \dot L^2_x} \les \|u\|_{L^5_t L^{10}_x}^5 \|u\|_{L^\infty_{t, x}}^{N-4} \les \|u_0\|_{\dot H^1}^5 \|u_0\|_{L^\infty}^{N-4}.
\ee
In particular, since $s_c-1 = 1/2-2/N$ and  
$3(\frac 1 2- \frac 2 N)+ 5(\frac 1 2 + \frac 2 N) = (N+1)4/N$, by interpolation between (\ref{eq1}) and (\ref{eq2}) we obtain (\ref{crit_2}).
\end{proof}

\section{Standard existence results}\lb{well_posedness}

We first state some standard Strichartz estimates, see \cite{keeltao}, that hold in scaling-invariant norms for equation (\ref{eq_sup}).

\begin{proposition}\lb{free_strichartz}
Consider a solution $u$ of the linear wave equation in three dimensions with a source term
$$
u_{tt}-\Delta u=F,\ u(0)=u_0,\ u_t(0)=u_1.
$$
Then
$$\begin{aligned}
\|u\|_{L^\infty_t \dot H^s_x \cap L^4_t \dot W^{s_c-1/2, 4}_x \cap L^{N/2}_t L^\infty_x} + \|u_t\|_{L^\infty_t \dot H^{s_c-1}_x \cap L^{4N/(N-4)}_t L^{4N/(N+4)}_x} \les \\
\les \|u_0\|_{\dot H^{s_c}} + \|u_1\|_{\dot H^{s_c-1}} + \|F\|_{L^1_t \dot H^{s_c-1}_x}.
\end{aligned}$$
\end{proposition}

Another simple linear estimate we shall use is
\be\lb{linfty}
\Big\|\frac {\sin(t\sqrt{-\Delta})}{\sqrt{-\Delta}} f\Big\|_{L^\infty} \les |t| \|f\|_{L^\infty}.
\ee

We next state some reversed-norm Strichartz estimates, following \cite{becgol}. Again we only state those estimates which hold in scaling-invariant norms for equation (\ref{eq_sup}).
\begin{proposition}\lb{reversed_strichartz} Consider a solution $u$ of the linear wave equation in three dimensions with a source term
$$
u_{tt} - \Delta u = F,\ u(0)=u_0,\ u_t(0)=u_1.
$$
Then
\be\lb{est_inf}\begin{aligned}
&\|u\|_{L^{3N/2, 2}_x L^\infty_t} \les \|u_0\|_{\dot H^{s_c}} + \|u_1\|_{\dot H^{s_c-1}} + \|F\|_{L^{\frac {3N}{2(N+1)}, 2}_x L^\infty_t},
\end{aligned}\ee
$$
\|u\|_{L^\infty_x L^{N/2}_t} \les \|u_0\|_{\dot H^{s_c}} + \|u_1\|_{\dot H^{s_c-1}} + \|F\|_{L^{3/2, 1}_x L^{N/2}_t}.
$$
\end{proposition}
Note that these reversed-norm estimates also hold (for the projection on the continuous spectrum) if the Hamiltonian is $-\Delta+V$ instead of $-\Delta$, where $V$ is a Kato-class potential, if there are no eigenvalues or resonances in the continuous spectrum.

\begin{observation} The following strictly stronger (in our context) inequalities are also true:
$$
\|D^{s_c-1}_x u\|_{L^{6, 2}_x L^\infty_t} \les \|u_0\|_{\dot H^{s_c}} + \|u_1\|_{\dot H^{s_c-1}} + \|D^{s_c-1}_x F\|_{L^{6/5, 2}_x L^\infty_t}
$$
and
$$
\|D^{s_c-1}_t u\|_{L^\infty_x L^2_t} \les \|u_0\|_{\dot H^{s_c}} + \|u_1\|_{\dot H^{s_c-1}} + \|D^{s_c-1}_t F\|_{L^{3/2, 1}_x L^2_t}.
$$
It is also possible to base a fixed point argument on these inequalities.
\end{observation}

Although we don't use them, we next state some standard well-posedness results for the semilinear wave equation (\ref{eq_sup})
$$
u_{tt}-\Delta u \pm |u|^N u = 0,\ u(0)=u_0,\ u_t(0)=u_1.
$$

The first existence result is one that holds in the standard Strichartz norms.
\begin{proposition} Assume that $N>4$ and $\|u_0\|_{\dot H^{s_c}} + \|u_1\|_{\dot H^{s_c-1}}$ is sufficiently small (or $N=4$ and the data are small and radially symmetric). Then equation (\ref{eq_sup}) admits a global solution $u$ with $(u_0, u_1)$ as initial data, such that
$$\begin{aligned}
\|u\|_{L^\infty_t \dot H^s_x \cap L^4_t \dot W^{s_c-1/2, 4}_x \cap L^{N/2}_t L^\infty_x} + \|u_t\|_{L^\infty_t \dot H^{s_c-1}_x \cap L^{4N/(N-4)}_t L^{4N/(N+4)}_x} \les \\
\les \|u_0\|_{\dot H^{s_c}} + \|u_1\|_{\dot H^{s_c-1}}.
\end{aligned}$$
In addition, $u$ scatters: there exist $(u_{0+}, u_{1+}) \in \dot H^{s_c} \times \dot H^{s_c-1}$ such that
$$
\lim_{t \to \infty} \|(u(t), u_t(t))-\Phi(t)(u_{0+}, u_{1+})\|_{\dot H^{s_c} \times \dot H^{s_c-1}} = 0
$$
and likewise as $t \to -\infty$.

More generally, if $(u_0, u_1) \in \dot H^{s_c} \times \dot H^{s_c-1}$ are not small, then there exist an interval $I=(-T, T)$ with $T>0$ and a solution $u$ to (\ref{eq_sup}) defined on $\R^3 \times I$, having $(u_0, u_1)$ as initial data, such that
$$\begin{aligned}
&\|u\|_{L^\infty_t \dot H^s_x (\R^3\times I) \cap L^4_t \dot W^{s_c-1/2, 4}_x (\R^3 \times I) \cap L^{N/2}_t L^\infty_x (\R^3 \times I)} + \\
&+ \|u_t\|_{L^\infty_t \dot H^{s_c-1}_x (\R^3 \times I) \cap L^{4N/(N-4)}_t L^{4N/(N+4)}_x (\R^3 \times I)} < \infty.
\end{aligned}$$
\end{proposition}
\begin{proof} This is a consequence of the standard Strichartz estimates for the free wave equation of \cite{keeltao}, see Proposition \ref{free_strichartz}.
	
The proof works by a contraction argument in the $L^\infty_t \dot H^s_x \cap L^4_t \dot W^{s_c-1/2, 4}_x \cap L^{N/2}_t L^\infty_x$ norm. Indeed, note that the nonlinearity can be bounded in the dual Strichartz norm by
$$
\||u|^N u\|_{L^1_t \dot H^{s_c-1}_x} \les \|u\|_{L^{N/2}_t L^\infty_x}^{N/2} \|u\|_{L^\infty_t L^{3N/2}_x}^{N/2} \|u\|_{L^\infty_t \dot H^s_x} \les \|u\|_{L^{N/2}_t L^\infty_x}^{N/2} \|u\|_{L^\infty_t \dot H^s_x}^{N/2+1}.
$$
See \cite{taylor} for mixed fractional Leibniz rules such as we are using here. A rather general statement is the following:
\begin{lemma} Let $1< p, p_1, p_2 <\infty$, $1 \leq q, q_1, q_2 \leq \infty$, $\alpha \in [0, 1]$, $D^\alpha$ be the Fourier multiplier $|\xi|^\alpha$, and $\frac 1 p_1 + \frac 1 p_2 = \frac 1 {\tilde p_1} + \frac 1 {\tilde p_2} = \frac 1 p$, $\frac 1 q_1 + \frac 1 q_2 = \frac 1 {\tilde q_1} + \frac 1 {\tilde q_2} = \frac 1 q$. Then
$$
\|D^\alpha(f g)\|_{L^{p, q}} \les \|D^\alpha f\|_{L^{p_1, q_1}} \|g\|_{L^{p_2, q_2}} + \|f\|_{L^{\tilde p_1, \tilde q_1}} \|D^\alpha g\|_{L^{\tilde p_2, \tilde q_2}}.
$$
\end{lemma}
This can be easily proved by complex interpolation between the $\alpha=0$ and $\alpha=1$ cases; see \cite{bergh}, p.\ 153.

Concerning scattering, we define
$$\begin{aligned}
u_{0+}&:=u_0-\int_0^\infty \frac{\sin(s\sqrt {-\Delta})}{\sqrt{-\Delta}} (|u(s)|^Nu(s)) \dd s,\\
u_{1+}&:=u_1+\int_0^\infty \cos(s\sqrt{-\Delta}) (|u(s)|^Nu(s)) \dd s.
\end{aligned}$$
Then
$$\begin{aligned}
u(t)=&\cos(t\sqrt{-\Delta}) u_0 + \frac{\sin(t\sqrt{-\Delta})}{\sqrt{-\Delta}} u_1 + \int_0^t \Big(\frac{\sin(t\sqrt{-\Delta})}{\sqrt{-\Delta}} \cos(s\sqrt{-\Delta}) - \\
&\cos(t\sqrt{-\Delta}) \frac{\sin(s\sqrt{-\Delta})}{\sqrt{-\Delta}}\Big) (|u(s)|^Nu(s)) \dd s,
\end{aligned}$$
so
$$\begin{aligned}
u(t)-\Phi_{0}(t)(u_{0+}, u_{1+}) =& \frac{\sin(t\sqrt{-\Delta})}{\sqrt{-\Delta}} \int_t^\infty \cos(s\sqrt{-\Delta}) (|u(s)|^Nu(s)) \dd s - \\
&\cos(t\sqrt{-\Delta}) \int_t^\infty \frac{\sin(s\sqrt{-\Delta})}{\sqrt{-\Delta}} (|u(s)|^Nu(s)) \dd s,
\end{aligned}$$
where $\Phi_{0}(t)$ is the first component of $\Phi(t)$. This expression goes to zero in the $\dot H^{s_c}$ norm. The same is true for $u_t$. 

In case the initial data is not small, the global $L^4_t \dot W^{s_c-1/2, 4}_x \cap L^{N/2}_t L^\infty_x$ Strichartz norm of its linear development is still finite, hence it becomes small on some sufficiently small interval $(-T, T)$, and we run the contraction argument on that interval. In the same way one can prove the uniqueness of the solution in $L^\infty_t \dot H^s_x (\R^3\times I) \cap L^{N/2}_t L^\infty_x (\R^3 \times I)$, where $I$ is the maximal interval on which the solution is defined.
\end{proof}

The second existence result holds in the reversed Strichartz norms introduced in \cite{becgol}, being a straightforward generalization of Proposition 5 from that paper.

\begin{proposition} Assume that $N \geq 4$ and $\|u_0\|_{\dot H^{s_c}}+\|u_1\|_{\dot H^{s_c-1}}$ is sufficiently small. Then equation (\ref{eq_sup}) admits a global solution $u$ with $(u_0, u_1)$ as initial data, such that
$$
\|u\|_{L^{3N/2}_x L^\infty_t \cap L^\infty_x L^{N/2}_t} \les \|u_0\|_{\dot H^{s_c}} + \|u_1\|_{\dot H^{s_c-1}}.
$$
Moreover, if the initial data $(u_0, u_1) \in \dot H^{s_c} \times \dot H^{s_c-1}$ are not small, there exist an interval $I=(-T, T)$ and a solution $u$ to (\ref{eq_sup}) on $\R^3 \times I$ such that
$$
\|u\|_{L^{3N/2}_x L^\infty_t(\R^3 \times I) \cap L^\infty_x L^{N/2}_t(\R^3 \times I)} < \infty.
$$
\end{proposition}
\begin{proof}
The proof is based on the reversed-norm Strichartz estimates of Proposition \ref{reversed_strichartz} and on a contraction argument in the $L^{3N/2}_x L^\infty_t \cap L^\infty_x L^{N/2}_t$ norm. Indeed, note that the nonlinearity can be bounded in the dual Strichartz norm by
$$\begin{aligned}
&\||u|^N u\|_{L^{\frac {3N}{2(N+1)}, 2}_x L^\infty_t} \les \|u\|_{L^{3N/2, 2}_x L^\infty_t}^{N+1},\\
&\||u|^N u\|_{L^{3/2, 1}_x L^{N/2}_t} \les \|u\|_{L^{3N/2, 2}_x L^\infty_t}^N \|u\|_{L^\infty_x L^{N/2}_t}.
\end{aligned}$$

One can obtain the large data local well-posedness result as follows: by means of smooth cutoffs, we restrict the initial data to sets on which their norm is small and solve the initial value problem with this restricted data. The solutions obtained will agree on some small time interval with the solution of the original problem due to the finite speed of propagation. Since there is a lower bound on how small the diameter of the sets is required to be, by piecing together all these partial solutions we obtain a global in space solution on some nonempty time interval.

In the same way one can prove that the solution is unique in $L^{3N/2}_x L^\infty_t (\R^3 \times I)$, where $I$ is a small interval (or $I=\R$ for small norm solutions).
\end{proof}

\section{Proof of the main results}\lb{proof_results}
\begin{proof}[Proof of Theorem \ref{thm_main}] This is a direct consequence of the standard existence theory, in view of the Strichartz estimates (\ref{crit}) and (\ref{crit_2}).

Explicitly, we write the solution $u$ as the sum of the free evolution of the outgoing initial data and a small perturbation, to which we apply a contraction argument. Let $u(x, t)=v(x, t)+w(x, t)$, where
$$
v_{tt} - \Delta v = 0,\ v(0)=v_0,\ v_t(0) = v_1
$$
and
\be\lb{eq_w}
w_{tt} - \Delta w \pm |v+w|^N(v+w) = 0,\ w(0) = w_0,\ w_t(0) = w_1.
\ee
We linearize equation (\ref{eq_w}) by writing it as
\be\lb{eq_w_liniar}
w_{tt} - \Delta w \pm |v+\tilde w|^N(v+\tilde w) = 0,\ w(0)=w_0,\ w_t(0)=w_1,
\ee
and solving for $w=F(\tilde w)$, while treating $\tilde w$ as given. The subsequent argument is the same if instead of null initial data we take small $(w_0, w_1) \in \dot H^{s_c} \times \dot H^{s_c-1}$ initial data in (\ref{eq_w}), i.e.\ if we allow for a small perturbation of the outgoing initial data $(v_0, v_1)$.

By a standard contraction argument we proceed to show that there exists $\tilde w$ of bounded $L^\infty_t \dot H^{s_c}_x \cap L^4_t \dot W^{s_c-1/2, 4}_x \cap L^{N/2}_t L^\infty_x$ norm such that $\tilde w=F(\tilde w)$.

The source term in equation (\ref{eq_w}) is $|v|^Nv$, which is controlled in the appropriate norm by (\ref{crit_2}): if we denote $$
K:= \|v_0\|_{\dot H^1}^{4/N} \|v_0\|_{L^\infty_{t, x}}^{1-4/N},
$$
then
\be\lb{simplu}
\||v|^Nv\|_{L^1_t \dot H^{s_c-1}_x} \les K^{N+1}.
\ee
The other (mixed) terms are bounded in the critical norm by (\ref{crit}) because for $N\leq 12$
\be\lb{condition}
s_c-1=1/2-2/N\leq 4/N.
\ee
Indeed,
$$\begin{aligned}
|v+\tilde w|^N(v+\tilde w)-|v|^Nv&=(|v+\tilde w|^N - |v|^N) (v + \tilde w)+|v|^N \tilde w\\
&=\tilde w\Big(\int_0^1 N |v+\alpha \tilde w|^{N-2}(v+\alpha \tilde w) \dd \alpha\Big)(v+\tilde w)+|v|^N\tilde w.
\end{aligned}$$
Furthermore, 
note that
$$
\|u^{N+1}\|_{L^1_t \dot H^{s_c-1}_x} \les \|u\|_{L^{N/2}_t L^\infty_x}^{N/2} \|u\|_{L^\infty_t L^{3N/2}_x}^{N/2} \|u\|_{L^\infty_t \dot W^{s_c-1, 6}} \les \|u\|_{L^{N/2}_t L^\infty_x}^{N/2} \|u\|_{L^\infty_t \dot W^{s_c-1, 6}}^{N/2+1}
$$
and more generally (here we use that $N$ is an integer, though it is probably unnecessary)
$$
\|u_1 \ldots u_{N+1}\|_{L^1_t \dot H^{s_c-1}_x} \les \|u_1\|_{L^\infty_t \dot W^{s_c-1, 6} \cap L^{N/2}_t L^\infty_x} \ldots \|u_{N+1}\|_{L^\infty_t \dot W^{s_c-1, 6} \cap L^{N/2}_t L^\infty_x}.
$$
Then for $0 \leq \alpha \leq 1$
\be\lb{product}\begin{aligned}
&\|\tilde w|v+\alpha \tilde w|^{N-2}(v+\alpha \tilde w)(v+\tilde w)\|_{L^1_t \dot H^{s_c-1}_x} \les \\
&\les \|\tilde w\|_{L^\infty_t \dot W^{s_c-1, 6} \cap L^{N/2}_t L^\infty_x} (\|v\|_{L^\infty_t \dot W^{s_c-1, 6} \cap L^{N/2}_t L^\infty_x}^N + \|\tilde w\|_{L^\infty_t \dot W^{s_c-1, 6} \cap L^{N/2}_t L^\infty_x}^N) \\
&\les \|\tilde w\|_{L^\infty_t \dot H^{s_c} \cap L^{N/2}_t L^\infty_x} (\|v\|_{L^\infty_t \dot W^{4/N, N/2}_x \cap L^{N/2}_t L^\infty_x}^N + \|\tilde w\|_{L^\infty_t \dot H^{s_c} \cap L^{N/2}_t L^\infty_x}^N),
\end{aligned}\ee
where $\dot H^{s_c} \subset \dot W^{s_c-1, 6}$ and $\dot W^{4/N, N/2} \subset \dot W^{s_c-1, 6}$. A similar estimate holds for $|v|^Nw$.

Note that, for (\ref{product}) to hold, each factor on the left-hand side must have at least $s_c-1$ derivatives. There are some monomials in (\ref{product}) with only one power of $v$; since all the other factors (powers of $w$) can only be bounded in scaling-invariant norms, due to scaling we must also bound $v$ in a scaling-invariant norm. Since by (\ref{crit}) $v$ only has $\frac 4 N$ derivatives in a scaling-invariant norm, condition (\ref{condition}) is necessary.

From (\ref{simplu}) and (\ref{product}), combined with standard Strichartz estimates, we get that
$$
\|w\|_{L^\infty_t \dot H^{s_c}_x\cap L^4_t\dot W^{s_c-1/2, 4}_x \cap L^{N/2}_t L^\infty_x} \les \|w_0\|_{\dot H^{s_c}} + \|w_1\|_{\dot H^{s_c-1}} + K^{N+1} + \|\tilde w\|_{L^\infty_t \dot H^{s_c}_x\cap L^{N/2}_t L^\infty_x}^{N+1}.
$$
If we assume that $\|\tilde w\|_{L^\infty_t \dot H^{s_c}_x\cap L^{N/2}_t L^\infty_x} \leq K$ and that $w_0$, $w_1$, and $K$ are sufficiently small, it follows that $\|w\|_{L^\infty_t \dot H^{s_c}_x\cap L^{N/2}_t L^\infty_x} \leq K$ as well.

Note that
$$\begin{aligned}
&|v+\tilde w_1|^N(v+\tilde w_1)-|v+\tilde w_2|^N(v+\tilde w_2) = \\ &= \int_0^1 \frac d {d\alpha} \big(|v+\alpha \tilde w_1+(1-\alpha) \tilde w_2|^N(v+\alpha \tilde w_1+(1-\alpha) \tilde w_2)\big) \dd \alpha \\
&= (\tilde w_1 - \tilde w_2) \int_0^1 (N+1) |v+\alpha \tilde w_1+(1-\alpha) \tilde w_2|^N \dd \alpha.
\end{aligned}$$
One then shows that, given two pairs $w^1$, $\tilde w^1$ and $w^2$, $\tilde w^2$ that both~fulfill~(\ref{eq_w_liniar}),
$$\begin{aligned}
&\|w^1-w^2\|_{L^\infty_t \dot H^{s_c}_x \cap L^4_t\dot W^{s_c-1/2, 4}_x \cap L^{N/2}_t L^\infty_x} \les \\
&\les \|\tilde w^1-\tilde w^2\|_{L^\infty_t \dot H^{s_c}_x \cap L^{N/2}_t L^\infty_x} (K^N+\|\tilde w^1\|_{L^\infty_t \dot H^{s_c}_x \cap L^{N/2}_t L^\infty_x}^N+\|\tilde w^2\|_{L^\infty_t \dot H^{s_c}_x \cap L^{N/2}_t L^\infty_x}^N).
\end{aligned}$$

It follows that the mapping $\tilde w \mapsto w$ is a contraction  in the sphere of radius $K$ in $L^\infty_t \dot H^{s_c}_x \cap L^{N/2}_t L^\infty_x$ when $w_0$, $w_1$, and $K$ are sufficiently small. The fixed point $w=\tilde w$ then gives rise to a global solution $u=v+w$ to (\ref{eq_sup}). As a byproduct we can also obtain the $L^\infty_t \dot H^{s_c-1}$ norm of $w_t$.

Estimate (\ref{disp_est}) is true because we can separately bound $v$ (by (\ref{crit})) and $w$ (by the fixed point argument) in the $L^{N/2}_t L^\infty_x$ norm.

Concerning scattering, define
$$\begin{aligned}
w_{0+}&:=w_0-\int_0^\infty \frac {\sin(s\sqrt{-\Delta})}{\sqrt{-\Delta}} (|u(s)|^Nu(s)) \dd s,\\
w_{1+}&:=w_1+\int_0^\infty \cos(s\sqrt{-\Delta}) (|u(s)|^Nu(s)) \dd s.
\end{aligned}$$
Since $|u(s)|^Nu(s) \in L^1_t \dot H^{s_c-1}_x$, it is easy to show (\ref{scatter}).

If the initial data are in $((\dot H^1 \cap L^\infty)\times L^2)_{out} + \dot H^s \times \dot H^{s-1}$, but not small, then still the norm $\|v\|_{L^{N/2}_t L^\infty_x} < \infty$ is finite, so there exists an interval $I=[0, T]$ on which $\|v\|_{L^{N/2}_t L^\infty_x(\R^3 \times I)}$ is small and same for the linear evolution of $(w_0, w_1)$. We then run the previous argument on this interval.

In the same manner one can prove the uniqueness of the solution in $L^\infty_t \dot H^{s_c}_x(\R^3 \times I) \cap L^{N/2}_t L^\infty_x(\R^3 \times I)$, where $I$ is the maximal interval of existence.
\end{proof}

\begin{proof}[Proof of Corollary \ref{cor_sup}] This follows from Theorem \ref{thm_main} and the $\R^3$ radial $\dot H^1$ Sobolev embedding
$$
|u_0(r)| \les r^{-1/2} \|u_0\|_{\dot H^1_{rad}}.
$$
Given that $u_0$ is supported outside the sphere $B(0, R)$, this embedding implies that
$$
\|u_0\|_{L^\infty} \les R^{-1/2} \|u_0\|_{\dot H^1_{rad}}.
$$
Therefore $\|u_0\|_{\dot H^1}^{4/N} \|u_0\|_{L^\infty}^{1-4/N} \les \|u_0\|_{\dot H^1} R^{-(1-4/N)/2}$. The conclusion follows by applying Theorem \ref{thm_main}.
\end{proof}

For the sake of completeness, we also state some local existence results. We begin with a simple, but weak result that holds for bounded initial data.
\begin{proposition}\lb{local1} Suppose that $N>0$, the initial data $(u_0, u_1)$ are radial and outgoing, and $u_0 \in L^\infty$. Then there exists a corresponding solution $u$ to (\ref{eq_sup}) on $\R^3 \times I$, $I=[0, T]$, such that $T \geq C\|u_0\|_{L^\infty}^{-N/2}$ and
$$
\|u\|_{L^\infty_{t, x}(\R^3 \times I)} \les \|u_0\|_{L^\infty}.
$$
\end{proposition}
Note that one cannot repeat this argument for later initial times because the nonlinearity generates incoming terms and for incoming initial data it is not enough for it to be in $L^\infty$.
\begin{proof}[Proof of Proposition \ref{local1}] We apply a fixed point argument. Linearize equation (\ref{eq_sup}) to
\be\lb{eq_tildeu}
u_{tt}-\Delta u \pm |\tilde u|^N \tilde u = 0,\ u(0)=u_0,\ u_t(0)=u_1.
\ee
Then, taking into account (\ref{sup}) and (\ref{linfty}),
$$
\|u\|_{L^\infty_{t, x}(\R^3 \times I)} \les \|u_0\|_{L^\infty} + T^2 \||\tilde  u|^N \tilde u\|_{L^\infty_{t, x}(\R^3 \times I)} \les \|u_0\|_{L^\infty} + T^2\|\tilde u\|_{L^\infty_{t, x}}^{N+1}.
$$
Thus, if $\|\tilde u\|_{L^\infty_{t, x}(\R^3 \times I)} \les \|u_0\|_{L^\infty}$ and $T \leq c \|u_0\|^{-N/2}_{L^\infty_{t, x}}$ with $c$ sufficiently small, we retrieve the same conclusion for $u$. In addition, the mapping $\tilde u \mapsto u$ is a contraction. Indeed, given two pairs $\tilde u^1$ and $u^1$, respectively $\tilde u^2$ and $u^2$, which fulfill (\ref{eq_tildeu}),
$$
\|u^1-u^2\|_{L^\infty_{t, x}(\R^3 \times I)} \les T^2 \|\tilde u^1 - \tilde u^2\|_{L^\infty_{t, x}} (\|\tilde u^1\|_{L^\infty_{t, x}}^N + \|\tilde u^2\|_{L^\infty_{t, x}}^N).
$$
Thus, if $T \leq c \|u_0\|^{-N/2}_{L^\infty_{t, x}}$ with sufficiently small $c$, then the mapping $\tilde u \mapsto u$ is a contraction on $\{u \mid \|u\|_{L^\infty_{t, x}(\R^3 \times I)} \les \|u_0\|_{L^\infty_{t, x}}\}$. The fixed point is a solution of (\ref{eq_sup}) with the desired properties.
\end{proof}

We next state another existence result in the subcritical sense, for large $((\dot H^1 \cap L^\infty) \times L^2)_{out} + (\dot H^2 \cap \dot H^1) \times (\dot H^1 \cap L^2)$ initial data (and in particular for $((\dot H^1 \cap L^\infty) \times L^2)_{out}$ initial data). The solution remains in the same class for some finite positive time and for all time if the initial data are small.
\begin{proposition}\lb{local2} Assume $N \geq 2$ and consider initial data $(u_0, u_1)=(v_0, v_1)+(w_0, w_1)$, where $(v_0, v_1) \in ((\dot H^1 \cap L^\infty) \times L^2)_{out}$ are radial and outgoing and $(w_0, w_1) \in (\dot H^2 \cap \dot H^1) \times (\dot H^1 \cap L^2)$ are radial. Then there exists a corresponding solution $u$ to (\ref{eq_sup}) on $\R^3 \times I$, where $I=[0, T]$ and
$$
T \geq C (\|w_0\|_{\dot H^2 \cap \dot H^1} + \|w_1\|_{\dot H^1 \cap L^2} + \|v_0\|_{\dot H^1 \cap L^\infty})^{-N},
$$
such that $u=v+w$ and
\be\lb{vvt}
\|(v, v_t)\|_{L^\infty_t ((\dot H^1 \cap L^\infty) \times L^2)_{out} (\R^3 \times I)} + \|v\|_{L^2_t L^\infty_x(\R^3 \times I)} \les \|v_0\|_{\dot H^1 \cap L^\infty},
\ee
$$\begin{aligned}
&\|(w, w_t)\|_{L^\infty_t (\dot H^2_x \cap \dot H^1_x \cap L^\infty_x \times \dot H^1_x \cap L^2_x) (\R^3 \times I)} \les \\
&\les \|w_0\|_{\dot H^2 \cap \dot H^1} + \|w_1\|_{\dot H^1 \cap L^2} + \|v_0\|_{\dot H^1 \cap L^\infty}. 
\end{aligned}$$
Assume in addition that $N \geq 4$ and
\be\lb{cond_subcritical}\begin{aligned}
&\|w_0\|_{\dot H^2 \cap \dot H^1} + \|w_1\|_{\dot H^1 \cap L^2} + \|v_0\|_{\dot H^1}^{5N-4} \|v_0\|_{L^\infty}^{(N-4)(N-1)} + \\
&+ \|v_0\|_{\dot H^1}^2 \|v_0\|_{L^\infty}^{N-2} + (\|w_0\|^{N-1}_{\dot H^2 \cap \dot H^1} + \|w_1\|^{N-1}_{\dot H^1 \cap L^2}) \|v_0\|_{\dot H^1} << 1.
\end{aligned}\ee
Then there exists a global solution $u$, forward in time, with this initial data, such that $u=v+w$, $v$ fulfills (\ref{vvt}), and
$$\begin{aligned}
\|(w, w_t)\|_{L^\infty_t(\dot H^2_x \cap \dot H^1_x \times \dot H^1_x \cap L^2_x)} + \|w\|_{L^2_t L^\infty_x} \les \|w_0\|_{\dot H^2 \cap \dot H^1} + \|w_1\|_{\dot H^1 \cap L^2} + \\
+ \|v_0\|_{\dot H^1}^5 \|v_0\|_{L^\infty}^{N-4} + \|v_0\|_{\dot H^1}^3 \|v_0\|_{L^\infty}^{N-2}.
\end{aligned}$$
\end{proposition}
\begin{proof}[Proof of Proposition \ref{local2}] As before, we write the solution as a sum of two terms, $u(x, t)=v(x, t)+w(x, t)$, where $v$ is the linear evolution of $(v_0, v_1)$ and $w$ is the contribution of $(w_0, w_1)$ and of the nonlinear terms:
$$
v_{tt}-\Delta v = 0,\ v(0)=v_0,\ v_t(0)=v_1
$$
and we linearize the second equation to (\ref{eq_w_liniar}), that is
$$
w_{tt}-\Delta w \pm |v+\tilde w|^N(v+\tilde w) = 0,\ w(0)=w_0,\ w_t(0)=w_1.
$$
Then clearly $(v, v_t)$ satisfy (\ref{vvt}), see (\ref{sup}) and Lemma \ref{equivalence}. In addition,
$$\begin{aligned}
&\|(w, w_t)\|_{L^\infty_t (\dot H^2_x \cap \dot H^1_x \times \dot H^1_x \cap L^2_x) (\R^3 \times I)} \les \\
&\les \|w_0\|_{\dot H^2 \cap \dot H^1} + \|w_1\|_{\dot H^1 \cap L^2} + T \||v+\tilde w|^N(v+\tilde w)\|_{L^\infty_t (\dot H^1_x \cap L^2_x)(\R^3 \times I)}.
\end{aligned}$$
Here
$$\begin{aligned}
\||v+\tilde w|^N(v+\tilde w)\|_{L^\infty_t L^2_x(\R^3 \times I)} \les \|v\|^{N-2}_{L^\infty_{t, x}(\R^3 \times I)} \|v\|^3_{L^\infty_t \dot H^1_x(\R^3 \times I)} + \\
+ \|\tilde w\|^{N-2}_{L^\infty_{t, x}(\R^3 \times I)} \|\tilde w\|^3_{L^\infty_t \dot H^1_x(\R^3 \times I)}
\end{aligned}$$
and
$$\begin{aligned}
\||v+\tilde w|^N(v+\tilde w)\|_{L^\infty_t \dot H^1_x(\R^3 \times I)} \les (\|v\|^N_{L^\infty_{t, x}(\R^3 \times I)} +\|\tilde w\|^N_{L^\infty_{t, x}(\R^3 \times I)}) \\
(\|v\|_{L^\infty_t \dot H^1_x(\R^3 \times I)} +  \|\tilde w\|_{L^\infty_t \dot H^1_x(\R^3 \times I)}).
\end{aligned}$$
Also note that
\be\lb{ineq_linfty}
\|\tilde w\|_{L^\infty_{t, x}(\R^3 \times I)} \les \|\tilde w\|_{L^\infty_t (\dot H^2_x \cap \dot H^1_x) (\R^3 \times I)}.
\ee
In conclusion, if
$$\|\tilde w\|_{L^\infty_t (\dot H^2_x \cap \dot H^1_x) (\R^3 \times I)} \les \|w_0\|_{\dot H^2 \cap \dot H^1} + \|w_1\|_{\dot H^1 \cap L^2} + \|v_0\|_{\dot H^1 \cap L^\infty}
$$
and if
$$
T \leq c (\|w_0\|_{\dot H^2 \cap \dot H^1} + \|w_1\|_{\dot H^1 \cap L^2} + \|v_0\|_{\dot H^1 \cap L^\infty})^{-N}
$$
with $c$ sufficiently small, then we retrieve the same conclusion for $w$. Under the same condition one can prove that the mapping $\tilde w \mapsto w$ is a contraction on the set $\{w \mid \|\tilde w\|_{L^\infty_t (\dot H^2_x \cap \dot H^1_x) (\R^3 \times I)} \les \|w_0\|_{\dot H^2 \cap \dot H^1} + \|w_1\|_{\dot H^1 \cap L^2} + \|v_0\|_{\dot H^1 \cap L^\infty}\}$. The fixed point $w$ gives rise to a solution $u=v+w$ with the required properties. In particular, we also retrieve a bound for $w_t$.

For the global existence result, we use the following estimates:
$$
\|(w, w_t)\|_{L^\infty_t (\dot H^1_x \times L^2_x)} + \|w\|_{L^2_t L^\infty_x} \les \|w_0\|_{\dot H^1} + \|w_1\|_{L^2} + \||v+\tilde w|^N(v+\tilde w)\|_{L^1_t L^2_x},
$$
where
$$
\||v+\tilde w|^N(v+\tilde w)\|_{L^1_t L^2_x} \les \|v\|_{L^\infty_t \dot H^1_x}^3 \|v\|_{L^2_t L^\infty_x}^2 \|v\|_{L^\infty_{t, x}}^{N-4} + \|\tilde w\|_{L^\infty_t \dot H^1_x}^3 \|\tilde w\|_{L^2_t L^\infty_x}^2 \|\tilde w\|_{L^\infty_{t, x}}^{N-4},
$$
together with
$$
\|(w, w_t)\|_{L^\infty_t (\dot H^2 \times \dot H^1)} \les \|w_0\|_{\dot H^2} + \|w_1\|_{\dot H^1} + \||v+\tilde w|(v+\tilde w)\|_{L^1_t \dot H^1_x},
$$
where
$$
\||v+\tilde w|(v+\tilde w)\|_{L^1_t \dot H^1_x} \les (\|v\|_{L^2_t L^\infty_x}^2 \|v\|_{L^\infty_{t, x}}^{N-2} + \|\tilde w\|_{L^2_t L^\infty_x}^2 \|\tilde w\|_{L^\infty_{t, x}}^{N-2})(\|v\|_{L^\infty_t \dot H^1_x} + \|\tilde w\|_{L^\infty_t \dot H^1_x}).
$$
Also note (\ref{ineq_linfty}) and that $\|v\|_{L^\infty_{t, x}} \leq \|v_0\|_{L^\infty}$ and $\|v\|_{L^\infty_t \dot H^1_x} \les \|v_0\|_{\dot H^1}$. It follows that whenever
\be\lb{condi}
\|\tilde w\|_{L^\infty_t (\dot H^2_x \cap \dot H^1_x) \cap L^2_t L^\infty_x} \leq \epsilon << 1
\ee
and
$$\begin{aligned}
&\|w_0\|_{\dot H^2 \cap \dot H^1} + \|w_1\|_{\dot H^1 \cap L^2} + \|v_0\|_{\dot H^1}^5 \|v_0\|_{L^\infty}^{N-4} + \|v_0\|_{\dot H^1}^3 \|v_0\|_{L^\infty}^{N-2} \les c\epsilon,\\
&\|v_0\|_{\dot H^1}^2 \|v_0\|_{L^\infty}^{N-2} + \epsilon^{N-1} \|v_0\|_{\dot H^1} << 1,
\end{aligned}$$
with $c$ sufficiently small (not depending on $\epsilon$), then we retrieve the same conclusion (\ref{condi}) for $w$.

In particular, for this to happen it is necessary that $(\|v_0\|_{\dot H^1}^5 \|v_0\|_{L^\infty}^{N-4})^{N-1} \|v_0\|_{\dot H^1} << 1$, which is part of our condition (\ref{cond_subcritical}).

Next, we prove that the mapping $\tilde w \mapsto w$ is a contraction. In the same manner as above it can be shown that, when $w^1$ and $\tilde w^1$, respectively $w^2$ and $\tilde w^2$, satisfy the linearized equation (\ref{eq_w_liniar}) and condition (\ref{condi}), then
$$
\|w^1-w^2\|_{L^\infty_t \dot H^1_x \cap L^2_t L^\infty_x} \les \|\tilde w^1-\tilde w^2\|_{L^\infty_t \dot H^1_x} (\|v_0\|_{\dot H^1}^4 \|v_0\|_{L^\infty}^{N-4} + \epsilon^N)
$$
and
$$
\|w^1-w^2\|_{L^\infty_t \dot H^2_x} \les \|\tilde w^1-\tilde w^2\|_{L^\infty_t (\dot H^2_x \cap \dot H^1_x)} (\|v_0\|_{\dot H^1}^2 \|v_0\|_{L^\infty}^{N-2} + \|v_0\|_{\dot H^1}^3 \|v_0\|_{L^\infty}^{N-3} + \epsilon^N).
$$
Thus, as long as $\epsilon$ is sufficiently small and
$$
\|v_0\|_{\dot H^1}^2 \|v_0\|_{L^\infty}^{N-2} + \|v_0\|_{\dot H^1}^3 \|v_0\|_{L^\infty}^{N-3} << 1,
$$
the mapping is a contraction. The ensuing fixed point $w$ gives rise to a solution $u$ of (\ref{eq_sup}) with the desired properties.

As part of the contraction argument we can also bound $w_t$. Putting together all the conditions we use, we obtain (\ref{cond_subcritical}).
\end{proof}

We continue with the proof of Theorem \ref{bounded_thm}, concerning global existence for bounded compact support initial data.
\begin{proof}[Proof of Theorem \ref{bounded_thm}]
This follows by a standard fixed point argument in the $L^{3N/2, 2}_x L^\infty_t \cap L^\infty_x L^{N/2}_t$ norm.

Let $u(x, t)=v(x, t)+w(x, t)$, where
$$
v_{tt}-\Delta v=0,\ v(0)=u_0, v_t(0)=u_1
$$
and
\be\lb{eqw}
w_{tt}-\Delta w+|v+w|^N(v+w)=0,\ w(0)=0,\ w_t(0)=0.
\ee
As in the proof of Theorem \ref{thm_main}, we linearize (\ref{eqw}) to
$$
w_{tt}-\Delta w\pm|v+\tilde w|^N(v+\tilde w)=0,\ w(0)=0,\ w_t(0)=0
$$
and then we prove by a contraction argument that there exists $\tilde w \in L^{3N/2, 2}_x L^\infty_t \cap L^\infty_x L^{N/2}_t$ for which $w=\tilde w$.

Note that by (\ref{reversed_est})
$$
\|v\|_{L^{3N/2, 2}_x L^\infty_t \cap L^\infty_x L^{N/2}_t} \les R^{2/N} \|u_0\|_{L^\infty}:=K.
$$
Then, since the initial data are zero,
$$\begin{aligned}
\|w\|_{L^{3N/2, 2}_x L^\infty_t \cap L^\infty_x L^{N/2}_t} &\les \||v+\tilde w|^N(v+\tilde w)\|_{L^{\frac{3N}{2(N+1)}, 2}_xL^\infty_t \cap L^{3/2, 1}_x L^{N/2}_t} \\
&\les K^{N+1}+\|\tilde w\|_{L^{3N/2, 2}_x L^\infty_t \cap L^\infty_x L^{N/2}_t}^{N+1}.
\end{aligned}$$
Thus, when $K$ is small, the mapping $\tilde w \mapsto w$ leaves a sufficiently small sphere in $L^{3N/2, 2}_x L^\infty_t \cap L^\infty_x L^{N/2}_t$ invariant.

Furthermore, considering two auxiliary functions $\tilde w_1$ and $\tilde w_2$ that give rise to solutions $w_1$, respectively $w_2$,
$$\begin{aligned}
&\|w_1-w_2\|_{L^{3N/2, 2}_x L^\infty_t \cap L^\infty_x L^{N/2}_t} \les \\
&\les \||v+\tilde w_1|^N(v+\tilde w_1) - |v+\tilde w_2|^N(v+\tilde w_2)\|_{L^{\frac{3N}{2(N+1)}, 2}_xL^\infty_t \cap L^{3/2, 1}_x L^{N/2}_t} \\
&\les \|\tilde w_1-\tilde w_2\|_{L^{3N/2, 2}_x L^\infty_t \cap L^\infty_x L^{N/2}_t} (K^N+\|\tilde w_1\|_{L^{3N/2, 2}_x L^\infty_t \cap L^\infty_x L^{N/2}_t}^N+\|\tilde w_2\|_{L^{3N/2, 2}_x L^\infty_t \cap L^\infty_x L^{N/2}_t}^N).
\end{aligned}$$
Thus the mapping $\tilde w \mapsto w$ is a contraction in a sufficiently small sphere when $K$ is also small. It therefore has a fixed point $w=\tilde w$, such that $u=v+w$ is a solution to (\ref{eq_sup}).
%
%
\end{proof}

We finally construct true large initial data global solutions to (\ref{eq_sup}).

\begin{proof}[Proof of Theorem \ref{large_data}] For $\alpha>0$ and $\epsilon<<1$, consider outgoing initial data $(u_0, u_1)$ supported on $\ov{B(0, 1+\epsilon) \setminus B(0, 1)}$, such that $u_0(r)=L\epsilon^{-\alpha}$ for $r \in [1, 1+\epsilon]$. Then $\|u_0\|_{L^\infty} \sim  L\epsilon^{-\alpha}$ and $\|u_0\|_{L^{1/\alpha}} \sim L$. Let $v$ be the linear evolution of $(u_0, u_1)$, that is
$$
v_{tt}-\Delta v =0,\ v(0)=u_0,\ v_t(0)=u_1.
$$
By (\ref{reversed_est'})
$$
\|v\|_{L^{3N/2, 2}_x L^\infty_t} \les \|u_0\|_{L^\infty} \sim L \epsilon^{-\alpha},\ \|v\|_{L^\infty_x L^{N/2}_t} \les \epsilon^{2/N} \|u_0\|_{L^\infty} \sim L \epsilon^{2/N-\alpha}.
$$
Letting $\epsilon$ go to zero, we cannot make the scaling-invariant $L^{3N/2, 2}_x L^\infty_t$ reversed Strichartz norm of $v$ small. This is why we examine one more iterate in the nonlinear contraction scheme.

Note that by Proposition \ref{formula} $v(r, t)$ is supported on $\ov{B(1+t+\epsilon) \setminus B(1+t)}$ and
\be\lb{vvv}
v(1+t+a, t) \les \frac {L\epsilon^{-\alpha}}{1+t+a} \leq \frac {L\epsilon^{-\alpha}}{1+t}.
\ee
Let $n > 2$. Therefore, $v^n(r=1+t+a, t)$ is supported on $\ov{B(1+t+\epsilon) \setminus B(1+t)}$ and bounded by $\frac {L^n\epsilon^{-n\alpha}}{(1+t+a)^n}$. To help with computations, we write this bound~as
$$
v^n(1+t+a, t) \les \int_0^\epsilon \frac {L^n\epsilon^{-n\alpha}}{(1+t+a)^{n-1}} \frac {\sin((1+t+a)\sqrt{-\Delta})}{\sqrt{-\Delta}} \delta_0 \dd a,
$$
where $\delta_0$ is Dirac's delta and we have taken advantage of the special form of the kernel of $\frac{\sin(t\sqrt{-\Delta})} {\sqrt{-\Delta}}(x, y)=\frac 1 {4\pi t} \delta_{|x-y|=t}$ for $t \geq 0$.

Let us estimate the Duhamel term
\be\lb{duhamel_term}
\int_0^t \frac {\sin((t-s)\sqrt{-\Delta})}{\sqrt{-\Delta}} v^n(s) \dd s.
\ee
Since we use absolute values, not cancellations, we bound this from above~by
$$\begin{aligned}
&\int_0^t \frac {\sin((t-s)\sqrt{-\Delta})}{\sqrt{-\Delta}} \frac {L^n\epsilon^{-n\alpha}}{(1+s)^{n-1}} \int_0^\epsilon \frac {\sin((1+s+a)\sqrt{-\Delta})}{\sqrt{-\Delta}} \delta_0 \dd a \dd s = \\
&=\int_0^\epsilon \int_0^t \frac {L^n\epsilon^{-n\alpha}}{(1+s)^{n-1}} \frac 1 2 \Big(\frac {\cos((t-1-2s-a)\sqrt{-\Delta})}{-\Delta} - \frac {\cos((1+t+a)\sqrt{-\Delta})}{-\Delta} \Big) \delta_0 \dd s \dd a.
\end{aligned}$$
Note that $\frac{\cos(t\sqrt{-\Delta})}{-\Delta}(x, y)=\frac 1 {4\pi|x-y|}\chi_{|x-y| \geq t}$. We obtain a bound of
\be\lb{computation}\begin{aligned}
&\int_0^\epsilon \int_0^t \frac {L^n\epsilon^{-n\alpha}}{(1+s)^{n-1}} \frac 1 r \chi_{[|t-1-2s-a|, 1+t+a]}(r) \dd s \dd a = \\
&= \int_0^\epsilon \int_{\max(0, \frac{t-1-a-r} 2)}^{\min(t, \frac{t-1-a+r}2)} \frac {L^n\epsilon^{-n\alpha}}{(1+s)^{n-1}} \dd s \frac {\chi_{[\max(0, 1+a-t), 1+t+a]}(r)} r \dd a \\
&\les \int_0^\epsilon \frac {\chi_{[\max(0, 1+a-t), 1+t+a]}(r)} r \Big(\frac {L^n \epsilon^{-n\alpha}} {(1+\max(0, \frac{t-1-a-r}2))^{n-2}} - \frac {L^n \epsilon^{-n\alpha}} {(1+\frac{t-1-a+r}2)^{n-2}}\Big)\dd a \\
&\les L^n \epsilon^{1-n\alpha} \chi_{[0, 2+t]}(r) \min\Big(\frac {1} r, 1 \Big),
\end{aligned}\ee
the second part using the mean value theorem.

It follows that
$$
\|(\ref{duhamel_term})\|_{\langle x \rangle^{-1} L^\infty_{t, x}} \les L^n \epsilon^{1-n\alpha}.
$$
Setting $n=N+1$, we obtain for example that this norm can be made arbitrarily small by letting $\epsilon$ go to zero if $\alpha<\frac 1 {N+1}$.

We write the solution $u$ as a sum of two parts, $u(x, t)=v(x, t)+w(x, t)$, where $v$ is the linear evolution of the initial data and $w$ is the contribution of the nonlinear terms:
\be\lb{auxiliary}
w_{tt}-\Delta w \pm (v+w)^{N+1} =0,\ w(0)=0,\ w_t(0)=0.
\ee
Recall that for simplicity we assumed that $N$ is even.

We then have to obtain similar bounds for the terms
\be\lb{duhamel_term'}
\int_0^t \frac {\sin((t-s)\sqrt{-\Delta})}{\sqrt{-\Delta}} (v^n(s) w^{N+1-n}(s)) \dd s
\ee
for $0 \leq n \leq N+1$. For $n>0$ we proceed in the same manner as in (\ref{computation}). Note that $v^n(t) w^{N+1-n}(t)$ is supported on $\ov{B(1+t+\epsilon) \setminus B(1+t)}$ and therefore has size
$$
v^n(r, t) w^{N+1-n}(r, t) \les \frac {L^n\epsilon^{-n\alpha}}{r^{N+1}} \|w\|^{N+1-n}_{\langle x \rangle^{-1} L^\infty_{t, x}}.
$$
In the same way as above we then obtain a bound of
$$
\|(\ref{duhamel_term'})\|_{\langle x \rangle^{-1} L^\infty_{t, x}} \les L^n\epsilon^{1-n\alpha} \|w\|^{N+1-n}_{\langle x \rangle^{-1} L^\infty_{t, x}}.
$$
For the last term corresponding to $n=0$, we use a different method, namely
$$
\Big\| \Big(\int_0^t \frac{\sin((t-s)\sqrt{-\Delta})}{\sqrt{-\Delta}} w^{N+1}(s) \dd s\Big)(x, t)\Big\|_{L^\infty_t} \les \int_{\R^3} \frac 1 {|x-y|} \|w(y, s)\|^{N+1}_{L^\infty_s} \dd y.
$$
Note that for $n>3$, by subdividing the integration domain into $|x-y| \leq \frac {|x|} 2$ and $|x-y| \geq \frac {|x|} 2$, we obtain
$$
\Big|\frac 1 {|x|} \ast \frac 1 {\langle x \rangle^n}\Big| = \int_{\R^3} \frac 1 {|x-y|} \frac 1 {\langle y \rangle^n} \dd y \les \frac 1 {\langle x\rangle}.
$$
Consequently, for $N+1>3$
$$
\Big\|\int_0^t \frac{\sin((t-s)\sqrt{-\Delta})}{\sqrt{-\Delta}} w^{N+1}(s) \dd s\Big\|_{\langle x \rangle^{-1} L^\infty_{t, x}} \les \|w\|_{\langle x \rangle^{-1} L^\infty_{t, x}}^{N+1}.
$$
We linearize equation (\ref{auxiliary}) to
\be\lb{linearized}
w_{tt}-\Delta w \pm (v+\tilde w)^{N+1} = 0,\ w(0)=0,\ w_t(0)=0.
\ee
Putting everything together, since this equation has null initial data, we have obtained that
$$
\|w\|_{\langle x \rangle^{-1} L^\infty_{t, x}} \les L^{N+1} \epsilon^{1-(N+1)\alpha} + L \epsilon^{1-\alpha} \|\tilde w\|^N_{\langle x \rangle^{-1} L^\infty_{t, x}} + \|\tilde w\|^{N+1}_{\langle x \rangle^{-1}L^\infty_{t, x}}.
$$
Assuming that $\|\tilde w\|_{\langle x \rangle^{-1}L^\infty_{t, x}} \leq \epsilon_0 << 1$, we retrieve the same for $w$ if we assume that $\epsilon$ is small enough and that $\alpha<\frac 1 {N+1}$.

Consider two pairs $w_1$ and $\tilde w_1$, respectively $w_2$ and $\tilde w_2$, which fulfill (\ref{linearized}). In the same manner as before we obtain that
$$\begin{aligned}
\|w_1-w_2\|_{\langle x \rangle^{-1} L^\infty_{t, x}} &\les \|\tilde w_1-\tilde w_2\|_{\langle x \rangle^{-1} L^\infty_{t, x}} (L^{N} \epsilon^{1-N\alpha} + L \epsilon^{1-\alpha} (\|\tilde w_1\|^{N-1}_{\langle x \rangle^{-1} L^\infty_{t, x}} + \\
&+ \|\tilde w_2\|^{N-1}_{\langle x \rangle^{-1} L^\infty_{t, x}}) + \|\tilde w_1\|^N_{\langle x \rangle^{-1}L^\infty_{t, x}} + \|\tilde w_2\|^N_{\langle x \rangle^{-1}L^\infty_{t, x}}).
\end{aligned}$$
We obtain that the mapping $\tilde w \mapsto w$ is a contraction on $\{w \mid \|w\|_{\langle x \rangle^{-1}L^\infty_{t, x}}\leq\epsilon_0\}$ if $\epsilon_0$ and $\epsilon$ are sufficiently small and if $\alpha<\frac 1 N$. Consequently it has a fixed point $w$ such that $u=v+w$ is a solution to (\ref{eq_sup}).

Since we want to obtain a dispersive solution, we shall also keep track throughout the contraction scheme of the $L^\infty_x L^1_t$ norm (in fact we can do better and we shall bound the $\langle x \rangle^{-1} L^\infty_x L^1_t$ norm). This is sufficient in view of the fact that
$$
\langle x \rangle^{-1} L^\infty_{t, x} \cap L^\infty_x L^1_t \subset L^{2N}_{t, x}.
$$
Note that
$$
\|w\|_{L^\infty_x L^1_t} \les \|(v+\tilde w)^{N+1}\|_{L^{3/2, 1}_x L^1_t}.
$$
However, this estimate is insufficient in view of the fact that $v$ is large. Returning to our computation (\ref{computation}), we extract some better bounds. Note that (\ref{duhamel_term}) is zero for $t \leq r-2$ and that for $t \geq r-2$
$$
(\ref{duhamel_term}) \les L^n \epsilon^{1-n\alpha} \min\Big(\frac 1 {r(1+\max(0, \frac{t-2-r} 2))^{n-2}}, \frac 1 {(1+\max(0, \frac{t-2-r} 2))^{n-1}} \Big)
$$
the latter by using the mean value theorem. It follows that for $n>3$
$$
\|(\ref{duhamel_term})\|_{\langle x \rangle^{-1} L^\infty_x L^1_t} \les L^n \epsilon^{1-n\alpha}.
$$
We again set $n=N+1$ and use a similar method (considering their support) to evaluate the terms (\ref{duhamel_term'}) for $0<n \leq N+1$, resulting in
$$
\|(\ref{duhamel_term'})\|_{\langle x \rangle^{-1} L^\infty_x L^1_t} \les L^n \epsilon^{1-n \alpha} \|w\|_{\langle x \rangle^{-1}L^\infty_{t, x}}^{N+1-n}.
$$
For the remaining term of the form (\ref{duhamel_term'}), in which $n=0$, i.e.\ for
\be\lb{duhamel_term''}
\int_0^t \frac {\sin((t-s)\sqrt{-\Delta})}{\sqrt{-\Delta}} w^{N+1}(s) \dd s,
\ee
we use the fact that for $N+1>3$
$$
\|(\ref{duhamel_term''})\|_{\langle x \rangle^{-1} L^\infty_x L^1_t} \les \|\tilde w^{N+1}\|_{\langle x \rangle^{-N-1} L^\infty_x L^1_t} \les \|\tilde w\|_{\langle x \rangle^{-1} L^\infty_{t, x}}^N \|\tilde w\|_{\langle x \rangle^{-1} L^\infty_x L^1_t}.
$$
In conclusion
$$
\|w\|_{\langle x \rangle^{-1} L^\infty_x L^1_t} \les L^{N+1} \epsilon^{1-(N+1)\alpha} + L \epsilon^{1-\alpha} \|\tilde w\|^N_{\langle x \rangle^{-1} L^\infty_{t, x}} + \|\tilde w\|_{\langle x \rangle^{-1} L^\infty_{t, x}}^N \|\tilde w\|_{\langle x \rangle^{-1} L^\infty_x L^1_t}.
$$
Thus the mapping $w \mapsto \tilde w$ takes the set $\{w \mid \|w\|_{\langle x \rangle^{-1} L^\infty_x L^1_t} \leq R,\ \|w\|_{\langle x \rangle^{-1}L^\infty_{t, x}}\leq\epsilon_0\}$ into itself for sufficiently large $R$ and sufficiently small $\epsilon$ and $\epsilon_0$.

Similarly we obtain that for two pairs $w_1$ and $\tilde w_1$, respectively $w_2$ and $\tilde w_2$, that satisfy (\ref{linearized}),
$$\begin{aligned}
&\|w_1-w_2\|_{\langle x \rangle^{-1} L^\infty_x L^1_t} \les \|\tilde w_1-\tilde w_2\|_{\langle x \rangle^{-1} L^\infty_{t, x}} (L^N \epsilon^{1-N\alpha} + L \epsilon^{1-\alpha} (\|\tilde w_1\|_{\langle x \rangle^{-1} L^\infty_{t, x}}^{N-1} + \\
&+ \|\tilde w_2\|_{\langle x \rangle^{-1} L^\infty_{t, x}}^{N-1})) + \|\tilde w_1-\tilde w_2\|_{\langle x \rangle^{-1} L^\infty_x L^1_t} (\|\tilde w_1\|_{\langle x \rangle^{-1} L^\infty_{t, x}}^N + \|\tilde w_2\|_{\langle x \rangle^{-1} L^\infty_{t, x}}^N).
\end{aligned}$$
It follows that the sequence $w_0=0$,
$$
w_{n+1}= \mp \int_0^t \frac {\sin((t-s)\sqrt{-\Delta})}{\sqrt{-\Delta}} (v(s)+w_n(s))^{N+1} \dd s
$$
converges in $\langle x \rangle^{-1} L^\infty_x L^1_t$ for sufficiently small $\epsilon$ and $\epsilon_0$ (in addition to $\langle x \rangle^{-1} L^\infty_{t, x}$, which we already knew).

In particular we can take $\alpha=\frac 1 {p_c}=\frac 2 {3N} < \frac 1 {N+1}$ so that $\|u_0\|_{L^{p_c}} \sim L$ is arbitrarily large.

The linear evolution $v$ of the initial data dominates all other terms, hence when estimating the norm of the solution it is enough to consider $v$, which is of size $L\epsilon^{-\alpha}$ in $L^\infty_t$, see (\ref{vvv}). 

\end{proof}
\begin{observation} The proof works more generally whenever $(u_0, u_1)$ are radial and outgoing, supported on $\ov{B(0, 1+\epsilon) \setminus B(0, 1)}$ and
$$
\epsilon^{1/(N+1)} \|u_0\|_{L^\infty} << 1.
$$
This means that the $L^{N+1}$ norm of $u_0$ must be small (though it does not vanish), but the $L^p(B(1, 2))$ norms for $p>N+1$ (in particular the $L^{p_c}$ norm) become arbitrarily high as $\epsilon \to 0$ for $p>N+1$.

\end{observation}

\begin{observation} A more interesting case should be taking large initial data supported on the union of two thin neighboring spherical shells, with opposite signs. This should lead to improved estimates due to cancellations.
\end{observation}

\appendix

\section{Using Choquet spaces in the study of the wave equation}

In this section we undertake a more detailed study of equation (\ref{eq_sup})
$$
u_{tt}-\Delta u \pm |u|^N u = 0,\ u(0) = u_0,\ u_t(0) = u_1
$$
with spread-out initial data of the form (\ref{farapart})
$$
u_0 = \sum_{j=1}^J \phi(x-y_j),\ u_1 = \sum_{j=1}^J \psi(x-y_j),\ |y_{j_1}-y_{j_2}| >> 1\ \forall j_1 \ne j_2
$$
by means of the Choquet integral.

The Choquet integral, introduced by Choquet in \cite{choquet}, is defined similarly to the Lebesgue integral, but is more general, in that it applies to outer measures (also to capacities). In some cases of interest, these outer measures do not give rise to a nontrivial $\sigma$-algebra of measurable sets, but we can still use the Choquet integral to integrate with respect to them.

\begin{definition}[See \cite{choquet} and \cite{adams}] An outer measure $\mu$ on a $\sigma$-algebra $\mathcal A \subset \mathcal P(A)$ is a function $\mu: \mathcal A \to [0, +\infty]$ such that:\\
	1) $\mu(\varnothing)=0$;\\
	2) Monotonicity: if $A_1 \subset A_2 \subset A$, then $\mu(A_1) \leq \mu(A_2)$;\\
	3) Subadditivity: for a countable family of sets $(A_n)_n \subset A$,
	$$
	\mu\bigg(\bigcup_n A_n\bigg) \leq \sum_n \mu(A_n).
	$$
\end{definition}

Then the Choquet integral of a nonnegative function $f:A \to \R$, $f \geq 0$, with respect to the outer measure $\mu$ is defined as
$$
\int_A f(x) d\mu := \int_0^\infty \mu(\{x \in A: f(x) \geq t\}) \dd t.
$$
The Choquet integral is in general not linear or even subadditive. Note, however, that if $\supp f \cap \supp g = \varnothing$ then
$$
\int f+g \dd \mu \leq \int f \dd \mu + \int g \dd \mu.
$$
It also has the following useful properties:\\
1) $\int \alpha f d\mu = \alpha \int f d\mu$;\\
2) $\int f d\mu = 0 \equiv f=0$ $\mu$-a.e.;\\
3) If $f \leq g$, then $\int f \dd \mu \leq \int g \dd \mu$.\\
Also, it is trivial to prove, using
$$
\{x: f(x)+g(x) \geq 2^k\} \subset \{x: f(x) \geq 2^{k-1}\} \cup \{x: g(x) \geq 2^{k-1}\},
$$
that
$$
\int f+g \dd \mu \leq 2\int f \dd \mu + 2\int g \dd \mu.
$$
A similar analysis shows that, since
$$
\{x: \sum_{k \geq 1} f_k(x) \geq 2^\ell\} \subset \bigcup_{k \geq 1} \{x: f_k \geq 2^{\ell-k}\},
$$
therefore
\begin{lemma}
$$
\int \bigg(\sum_k f_k\bigg) \dd\mu \leq 2\int f_1 \dd\mu + 4 \int f_2 \dd \mu +\ldots,
$$
\end{lemma}
\noindent meaning that any geometric series with ratio less than $1/2$ converges.

Also note that
$$
\int f_1 + \ldots + f_N \dd \mu \leq N \bigg(\int f_1 \dd \mu + \ldots + \int f_N \dd \mu\bigg).
$$
Since the right hand side constant grows with the number of terms, in general it may be difficult to sum an infinite series.

\begin{definition}
	For $1\leq p<\infty$, let the Choquet space $L^p(\mu)$ be the space of functions such that
	$$
	\|f\|_{L^p(\mu)}^p := \int_A |f(x)|^p \dd \mu < \infty,
	$$
	with $L^\infty(\mu)$ also defined using the essential supremum with respect to $\mu$.
\end{definition}

In general $\|f\|_{L^p(\mu)}$ is only a quasinorm, not a norm (except for $p=\infty$). A quick computation, based on Newton's binomial formula, shows that, when $p \geq 1$ is an integer,
$$
\|f+g\|_{L^p(\mu)} \leq (p+1)^{1/p} (\|f\|_{L^p(\mu)} + \|g\|_{L^p(\mu)}).
$$
Thus, a geometric series with ratio less than $(p+1)^{-1/p}$ converges in $L^p(\mu)$ when $p \geq 1$ is an integer (and less than $\lfloor p+1 \rfloor^{-1/\lfloor p \rfloor}$ in general).

A quasinorm (raised to a suitable power) induces a metric structure, see \cite{bergh}, so $L^p(\mu)$ are also metric spaces. Consider any Cauchy sequence in $L^p(\mu)$; one can extract a subsequence such that the difference of successive terms has a small ratio, so it converges. Hence $L^p(\mu)$ is a complete metric space, for $1 \leq p \leq \infty$.

In some cases there exist equivalent norms, so the spaces are normable (Banach spaces). However, although we shall point out when equivalent norms exist, quasinorms are also adequate for our purpose, see below.

We want a norm that accounts for the fact that initial data of the type (\ref{farapart}) are locally small and spread out, converts this sparseness into smallness, and within which we can close the loop in a fixed point argument for equation (\ref{eq_sup}).

By necessity, such a norm cannot be invariant under symmetric rearrangement, since symmetric rearrangement makes the initial data (\ref{farapart}) large.

We start with the Kato-type norm, introduced in \cite{rod} and \cite{goldberg},
$$
\|f\|_{\mc K_\alpha} := \sup_{y \in \R^3} \||x-y|^{-\alpha} f(x)\|_{L^1_x},
$$
where $\alpha \in [0, 3)$. More generally, for $1 \leq p < \infty$, $1 \leq q \leq \infty$, let $\mathcal K_{\alpha, p} := \mathcal K_{\alpha, p, p}$, where
$$
\|f\|_{\mc K_{\alpha, p, q}} := \sup_{y \in \R^3} \||x-y|^{-\alpha} f(x)\|_{L^{p, q}_x}.
$$
Since $L^{p, q}$ are normed spaces for $1<p<\infty$, $1 \leq q \leq \infty$, it follows that $\mathcal K_{\alpha, p, q}$ are also normed spaces in the same range (together with $\mathcal K_\alpha$ and $\mathcal K_{0, \infty}$).

Restricted to characteristic functions of Lebesgue measurable sets (or even Borel sets), the $\mathcal K_\alpha$ norm gives an outer measure
$$
\mu_\alpha(A) := \|\chi_A(x)\|_{\mc K_\alpha}.
$$
It is easy to check that there are no nontrivial measurable sets for this outer measure (except for $\alpha=0$, when $\mu_0$ is the usual Lebesgue outer measure), so we need to use the Choquet integral.

The outer measures $\mu_\alpha$ constitute a different class from previously studied examples in the context of the Choquet integral, such as \cite{choquet}, \cite{adams} (in the context of capacity theory), or \cite{doth} (tent spaces).

We next establish the properties of the quasinormed Choquet spaces $L^p(\mu_\alpha)$, $1 \leq p<\infty$, to which we add $L^\infty(\mu_\alpha) := L^\infty$. Real interpolation (see \cite{bergh}, Chapter 3) works for these spaces in the same manner as for Lebesgue spaces, namely
$$
(L^\infty, L^1(\mu_\alpha))_{(\theta, q)} = L^{1/\theta, q}(\mu_\alpha).
$$
We obtain a larger family of quasinormed Lorentz--Choquet spaces $L^{p, q}(\mu_\alpha)$, $1 \leq p \leq \infty$, $1 \leq q \leq \infty$.

The $K$ functional (see \cite{bergh}, p.\ 38) has the same form as for the usual Lorentz spaces, namely
\begin{lemma}
$$
K(f, t, L^1(\mu_\alpha), L^\infty) = \int_0^t f_\alpha^*(s) \dd s,
$$
where $f_\alpha^*$ is the decreasing rearrangement of $f$, but with respect to $\mu_\alpha$:
$$
f_\alpha^*(s) = \inf \{ u \geq 0: \mu_\alpha(\{x: |f(x)| \geq u\}) < s\}.
$$
\end{lemma}
Thus the $L^{p, q}(\mu_\alpha)$ quasinorms also have the usual definition:
\begin{definition} For $1 \leq p < \infty$, $1 \leq q \leq \infty$,
	$$
	\|f\|_{L^{p, q}(\mu_\alpha)} = \begin{cases} 
	\bigg( \ds\int_0^{\infty} (t^{\frac{1}{p}} f_\alpha^{\ast}(t))^q\, dt/t \bigg)^{\frac{1}{q}} & q \in (0, \infty), \\
	\sup\limits_{t > 0} \, t^{\frac{1}{p}}  f_\alpha^{\ast}(t)   & q = \infty.
	\end{cases}
	$$
\end{definition}

Therefore the spaces $L^{p, q}(\mu_\alpha)$ have some of the usual properties of Lorentz spaces, such as $L^{p, p}(\mu_\alpha) = L^p(\mu_\alpha)$ and $L^{p, q_1}(\mu_\alpha) \subset L^{p, q_2}(\mu_\alpha)$ for $q_1 \leq q_2$. As usual, the spaces $L^{\infty, q}(\mu_\alpha)$, $q<\infty$, behave differently and we won't study them.

We next define an $L^p(\mu_\alpha)$ atom.

\begin{definition} For $1 \leq p < \infty$, we say that $a$ is an $L^p(\mu_\alpha)$ atom if $a$ is essentially bounded, has a support of finite $\mu_\alpha$ outer measure, and is $L^p(\mu_\alpha)$-normalized, i.e.
	$$
	\|a\|_{L^\infty}^p \cdot \mu_\alpha(\supp a) = 1.
	$$
\end{definition}

It is useful to extend the following simple atomic characterization, developed in \cite{bec_new_schroedinger} for the usual Lorentz spaces, to these Lorentz--Choquet spaces:
\begin{proposition}\lb{atomic} Fix $1 \leq p < \infty$, $1 \leq q \leq \infty$. Then $f \in L^{p, q}(\mu_\alpha)$ if and only if it admits an atomic decomposition
$$
f = \sum_{k \in \Z} c_k a_k,
$$
where $c_k \in \R$, each atom $a_k$ has size $\|a_k\|_{L^\infty} \sim 2^k$, their supports are pairwise disjoint, and then $\|f\|_{L^{p, q}(\mu_\alpha)}^q \sim \sum_k |c_k|^q$ (or $\|f\|_{L^{p, \infty}(\mu_\alpha)} \sim \sup_k |c_k|$ when $q=\infty$).
	
\end{proposition}
For each $x \in \R^3$ at most one term is nonzero. The sum above is interpreted in this pointwise finite sense, but clearly it also converges in the $L^{p, q}(\mu_\alpha)$ norm (if there is one) unless $q=\infty$. For the converse, it is not necessary that the supports of the atoms $\alpha_k$ should be pairwise disjoint. The proof is identical to the one in \cite{bec_new_schroedinger}.

We next establish the relation between the quasinorm of $L^{p, q}(\mu_\alpha)$ and the norm of $\mathcal K_{\alpha, p, q}$. There is a clear relation for functions localized in height:
\begin{lemma}\lb{limited_height} Let $1 \leq p < \infty$, $1 \leq q_1, q_2 \leq \infty$ and $f \in L^{p, q_1}(\mu_\alpha)$ be such that $M \leq |f| \leq 2M$ almost everywhere. Then
$$
\|f\|_{L^{p, q_1}(\mu_\alpha)} \sim \|f\|_{\mathcal K_{\alpha/p, p, q_2}} \sim M \mu_\alpha(\supp f)^{1/p},
$$
with bounds independent of $M>0$.
\end{lemma}

Obviously $L^1(\mu_\alpha) \subset \mathcal K_\alpha$ and at the other endpoint $L^\infty = \mathcal K_{0, \infty}$. Straight from the definition, we then have that for $1\leq p \leq \infty$
$$
L^p(\mu_\alpha) \subset \mathcal K_{\alpha/p, p}.
$$

Using the atomic decomposition above, we more generally obtain that
\begin{lemma} The quasinorm of $L^{p, q}(\mu_\alpha)$ can be expressed as
$$
\|f\|_{L^{p, q}(\mu_\alpha)} \sim \bigg(\sum_{k \in \Z} \|\chi_{|f(x)| \in [2^{k-1}, 2^k)}(x) f(x)\|_{\mathcal K_{\alpha/p, p}}^q\bigg)^{1/q}
$$
for $q<\infty$ and
$$
\|f\|_{L^{p, \infty}(\mu_\alpha)} \sim \sup_{k \in \Z} \|\chi_{|f(x)| \in [2^{k-1}, 2^k)}(x) f(x)\|_{\mathcal K_{\alpha/p, p}}
$$
for $q=\infty$. Consequently, for $1 \leq p, q \leq \infty$,
\be\lb{embedding2}
L^{p, q}(\mu_\alpha) \subset \mathcal K_{\alpha/p, p, q} \subset L^{p, \infty}(\mu_\alpha).
\ee
\end{lemma}
\begin{proof}
The equivalence of the quasinorms follows from Lemmas \ref{atomic} and \ref{limited_height}.
	
The first inclusion is obvious. For the second one, let $f \in \mathcal K_{\alpha/p, p, q}$ and
$$
a_k = \chi_{|f(x)| \in [2^{k-1}, 2^k)}(x) f(x).
$$
Then, by Lemma \ref{limited_height}, $\|a_k\|_{L^p(\mu_\alpha)} \sim \|a_k\|_{\mathcal K_{\alpha/p, p, q}} \leq \|f\|_{\mathcal K_{\alpha/p, p, q}}$. We conclude that
$$
\|f\|_{L^{p, \infty}(\mu_\alpha)} \sim \sup_k \|a_k\|_{\mathcal K_{\alpha/p, p}} \leq \|f\|_{\mathcal K_{\alpha/p, p, q}}.
$$
\end{proof}

\begin{observation} 1. In particular, (\ref{embedding2}) means that $\mathcal K_{\alpha/p, p, \infty} = L^{p, \infty}(\mu_\alpha)$, with equivalent quasinorms. Therefore $L^{p, \infty}(\mu_\alpha)$ is normable (and a Banach space) for $1<p \leq \infty$.
	
This makes $L^{p, \infty}(\mu_\alpha)$ convenient to work with, but we shall also need other values of $q$, in particular $L^{p, 1}(\mu_\alpha)$.

For any $1<p<\infty$, take $1<p_1<p<p_2<\infty$; then by interpolating again (see the reiteration theorems in \cite{bergh}, Theorem 3.5.3, p.\ 50 and Theorem 3.11.5, p.\ 67, as well as the discussion on p.\ 63) we get
$$
(L^{p_1, \infty}(\mu_\alpha), L^{p_2, \infty}(
\mu_\alpha))_{\theta, q} = L^{p, q}(\mu_\alpha),
$$
where $p=(1-\theta)p_1+\theta p_2$, $\theta \in (0, 1)$. Interpolating between two Banach spaces we are bound to obtain another Banach space, not a quasinormed space. Thus, all spaces $L^{p, q}(\mu_\alpha)$, $1<p<\infty$, $1 \leq q \leq \infty$ are Banach spaces, i.e.\ they possess norms equivalent to the original quasinorms. The norms are obtained by interpolation and are not explicit at this point, but one could extract an explicit formula, see \cite{bergh}.
	
2. The spaces $L^{p, q}(\mu_\alpha)$ and $\mathcal K_{\alpha/p, p, q}$ are invariant under translation and rescaling and have the same scaling as $L^{3p/(3-\alpha)}$, namely
$$
\|f(c x)\|_{L^{p, q}_x(\mu_\alpha)} = c^{(\alpha-3)/p} \|f\|_{L^{p, q}(\mu_\alpha)}.
$$
	
3. One of the more useful properties of these Lorentz--Choquet spaces, which follows trivially from the definition, is that $(L^{p, q}(\mu_\alpha))^r = L^{p/r, q/r}(\mu_\alpha)$. 
More generally, the usual H\"{o}lder's inequality holds for fixed $\alpha$ (and it has a generalization for varying $\alpha$).
	
4. We now have several ways of approaching the convergence of Cauchy sequences. For example, (\ref{embedding2}) means that a Cauchy sequence in $L^{p, q}(\mu_\alpha)$ converges in $K_{\alpha/p, p, q}$ and in $L^{p, \infty}(\mu_\alpha)$, which are both Banach spaces for $p>1$. Furthermore, also for $1<p<\infty$, $L^{p, q}(\mu_\alpha)$ is a Banach space, so a Cauchy sequence converges. Finally, $L^1(\mu_\alpha)$ is not a Banach space, but it is a quasinormed space, so Cauchy sequences converge by the argument mentioned above.

	
The sums of successive iterates, $\sum_{k=1}^n T^k u_0$, produced by a contraction mapping $T$ on $L^{p, q}(\mu_\alpha)$, are not guaranteed to form a Cauchy series in the same space, but will converge due to (\ref{embedding2}) in $K_{\alpha/p, p, q}$ and $L^{p, \infty}(\mu_\alpha)$, which are Banach spaces. If the contraction ratio is sufficiently small, then this will be a Cauchy series and will converge in the original quasinorm.
	
\end{observation}

Next, we establish a form of Young's inequality for these Lorentz--Choquet spaces. We first prove a weak-type inequality in some extreme cases, then we upgrade the result by interpolation.

\begin{lemma}[Young's inequality]\lb{young} Let $f \in L^{p, q}(\mu_\alpha)$, $1 < p < \infty$, $1 \leq q \leq \infty$ or $(p, q)\in\{(1, 1), (\infty, \infty)\}$, and $g \in L^1$. Then $f \ast g \in \mathcal K_{\alpha/p, p, q}$ and
	\be\lb{young_0}
	\|f \ast g\|_{L^{p, \infty}(\mu_\alpha)} \leq \|f \ast g\|_{\mathcal K_{\alpha/p, p, q}} \leq \|f\|_{L^{p, q}(\mu_\alpha)} \|g\|_{L^1}.
	\ee
	Moreover,
	\be\lb{young_1}
	\|f \ast g\|_{L^{p, q}(\mu_\alpha)} \les \|f\|_{L^{p, q}(\mu_\alpha)} \|g\|_{L^1}
	\ee
	for $1<p<\infty$, $1 \leq q \leq \infty$. If $1 \leq p, q \leq \infty$, then
	\be\lb{young_2}
	\|f \ast g\|_{L^\infty} \leq \|f\|_{L^{p, q}(\mu_\alpha)} \|g\|_{|y|^{-\alpha/p} L^{p', q'}_y}.
	\ee
Finally, for $1/p=1/p_1+1/p_2-1$, $1 < p, p_1 \leq \infty$, $1 \leq p_2 \leq \infty$,
\be\lb{young3}
\|f \ast g\|_{L^{p, \infty}} \leq \|f\|_{L^{p_1, \infty}(\mu_\alpha)} \|g\|_{|y|^{-\alpha(1-1/p_2)} L^{p_2, 1}_y}.
\ee
\end{lemma}
This is probably not a complete list of cases in which Young's inequality is valid, but it suffices for our purposes.


\begin{proof} The first inequality in (\ref{young_0}) comes from (\ref{embedding2}). Concerning the second inequality, for fixed $x_0 \in \R^3$
\be\lb{mink}
\begin{aligned}
\bigg\||x-x_0|^{-\alpha/p} \int_{\R^3} f(x-y) g(y) \dd y\bigg\|_{L^{p, q}_x} &\leq \int_{\R^3} \sup_{y \in \R^3} \||x-x_0|^{-\alpha/p} f(x-y)\|_{L^{p, q}_x} |g(y)| \dd y \\
&= \|f\|_{\mathcal K_{\alpha/p, p, q}} \|g\|_{L^1} \leq \|f\|_{L^{p, q}(\mu_\alpha)} \|g\|_{L^1}.
\end{aligned}
\ee
Here we used Minkowski's inequality. Since (\ref{mink}) holds uniformly for $x_0 \in \R^3$, we have proved the second inequality in (\ref{young_0}).
	
	By real interpolation, we can strengthen this to (\ref{young_1})
	$$
	\|f \ast g\|_{L^{p, q}(\mu_\alpha)} \les \|f\|_{L^{p, q}(\mu_\alpha)} \|g\|_{L^1}
	$$
	for $1<p<\infty$, $1 \leq q \leq \infty$ (i.e.\ everything except the endpoints), at the price of a constant.
	
	When $q=\infty$ and $1<p<\infty$ (\ref{young_1}) follows directly from Minkowski's inequality, since $L^{p, \infty}(\mu_\alpha)$ are Banach spaces with norms invariant under translation.
	
	The proof of (\ref{young_2}) is based on duality:
	$$
	\|f \ast g\|_{L^\infty} \leq \|f\|_{\mathcal K_{\alpha/p, p, q}} \|g\|_{|y|^{-\alpha/p} L^{p', q'}_y} \leq \|f\|_{L^{p, q}(\mu_\alpha)} \|g\|_{|y|^{-\alpha/p} L^{p', q'}_y}.
	$$
	
	Finally, (\ref{young3}) is proved by complex interpolation, which we can use here because all the spaces are Banach spaces.
\end{proof}

We next use Young's inequality to obtain fractional integration bounds.

\begin{proposition}[Fractional integration]\lb{fractional}
	$$
	\|f \ast |x|^{-\beta}\|_{L^{r, q}(\mu_\alpha)} \les \|f\|_{L^{p, q}(\mu_\alpha)},
	$$
	where $0 \leq \alpha<3$, $1 < p < r<\infty$, $1 \leq q \leq \infty$, and $(1-\alpha/3)/r = (1-\alpha/3)/p + \beta/3 - 1$.
\end{proposition}
\begin{proof}
	One endpoint we use for interpolation is (\ref{young_1}):
	$$
	\|f \ast g\|_{L^{p, q}(\mu_\alpha)} \les \|f\|_{L^{p, q}(\mu_\alpha)} \|g\|_{L^1}
	$$
	The other endpoint is (\ref{young_2}):
	$$
	\|f \ast g\|_{L^\infty} \leq \|f\|_{\mathcal K_{\alpha/p, p, q}} \|g\|_{|y|^{-\alpha/p} L^{p', q'}_y} \leq \|f\|_{L^{p, q}(\mu_\alpha)} \|g\|_{|y|^{-\alpha/p} L^{p', q'}_y}.
	$$
	
	Since we are dealing with quasinorms, we can only use real interpolation, but the real interpolates of $L^1$ and $|y|^{-\alpha/p} L^{p', q'}$ are in general badly behaved spaces (because we are changing both the exponent and the measure). However, at this point we are no longer interested in optimal conditions for $g$, since we only need to perform fractional integration.
	
Define the following $L^\infty$- and dyadic partition-based family of spaces:
$$
2^{-sk} \ell^p_k(L^\infty_y) := \{g: 2^{sk} \|\chi_{|y| \in [2^{k-1}, 2^k)}(y) g(y)\|_{L^\infty_y} \in \ell^p_k\}.
$$
By rescaling, we identify the set of bounded functions on any dyadic annulus, $L^\infty(|y| \in [2^{k-1}, 2^k))$, with $A:=L^\infty(|y| \in [1/2, 1))$. In other words, the mapping
$$
T: 2^{-sk} \ell^p_k A \to 2^{-sk} \ell^p_k(L^\infty_y),\ T((a_k)_k) := \sum_k a_k(x/2^k) 
$$
is an isomorphism, where
$$
2^{-sk} \ell^p_k A := \{(a_k)_k : 2^{sk} \|a_k\|_A \in \ell^p_k\}.
$$
This second formulation is more suitable for interpolation.
	
	Note that
	$$
	2^{-3k} \ell^1_k(L^\infty_y) \subset L^1_y
	$$
	and
	$$
	2^{-(\alpha/p+3/p')k} \ell^\infty_k(L^\infty_y) \subset |y|^{-\alpha/p} L^{p', \infty}_y.
	$$
	
Using these more restrictive spaces, we rewrite (\ref{young_1}) and (\ref{young_2}) as
$$
\|f \ast g\|_{L^{p, 1}(\mu_\alpha)} \les \|f\|_{L^{p, 1}(\mu_\alpha)} \|g\|_{2^{-3k} \ell^1_k(L^\infty_y)}
$$
	and
	$$
	\|f \ast g\|_{L^\infty} \les \|f\|_{L^{p, 1}(\mu_\alpha)} \|g\|_{2^{-(\alpha/p+3/p')k} \ell^\infty_k(L^\infty_y)}.
	$$
	Thus, we can apply Theorem 5.6.1 from \cite{bergh} and obtain that
	$$
	\|f \ast g\|_{L^{r, \infty}(\mu_\alpha)} \leq \|f\|_{L^{p, 1}(\mu_\alpha)} \|g\|_{2^{-\beta k} \ell^\infty_k (L^\infty_y)},
	$$
	where the relation between $r$, $p$, and $\beta$ is dictated by scaling. In other words, we have proved that fractional integration takes $L^{p, 1}(\mu_\alpha)$ to $L^{r, \infty}(\mu_\alpha)$. By using real interpolation again, for $f$ this time, we obtain the $L^{p, q}(\mu_\alpha) \mapsto L^{r, q}(\mu_\alpha)$ boundedness everywhere except at the endpoints.
\end{proof}
The optimal condition on $g$ involves an $L^{p', q'}$-based dyadic decomposition, instead of an $L^\infty$-based one. However, we do not need such a sharp statement.

%
%
%

Now we have the tools needed for a contraction-based solution to our problem. The point is that we can use a contraction argument and bootstrap in the $L^{p, \infty}_x(\mu_\alpha) L^\infty_t$ norm below, for sufficiently small/sparse initial data.
\begin{proposition}[Small data global well-posedness]\lb{closed_loop}
Assume that $N>4$ and take $N+1<p<3N/2$, $\alpha = 3-2p/N$. Then initial data of the form (\ref{farapart}), for sufficiently small $(\phi, \psi) \in \dot B^{s_c}_{2, \infty} \times \dot B^{s_c-1}_{2, \infty}$ and sufficiently large $|y_{j_1}-y_{j_2}|$, lead to a small global solution of (\ref{eq_sup}) in $L^{p, \infty}_x(\mu_\alpha) L^\infty_t$.
\end{proposition}
The range of $N$ in this statement is not optimal. Note that $p>N+1$ is equivalent to $\alpha<1-2/N$, so $\alpha>0$ only requires $N>2$. However, the necessary Strichartz-type inequalities were only proved in \cite{becgol} for the range $N \geq 4$ (and here we interpolated again to get Besov spaces, which excludes the endpoint $N=4$).
\begin{proof}
	The linear evolution of the initial data (\ref{farapart}) is small in the $L^{p, \infty}(\mu_\alpha) L^\infty_t$ norm, where $\alpha=3-2p/N$ is dictated by scaling. Indeed, for $\alpha>0$ and sufficiently far apart centers,
	$$
	\|\Phi_0(u_0, u_1)\|_{L^{p, \infty}_x(\mu_\alpha) L^\infty_t} \les \|\Phi_0(\phi, \psi)\|_{L^{p, \infty}_x(\mu_\alpha) L^\infty_t} + \epsilon.
	$$
	But the linear evolution of each bump is small in $L^{3N/2, \infty}_x L^\infty_t \subset L^{p, \infty}_x(\mu_\alpha) L^\infty_t$.
	
Here $L^{3N/2, \infty} \subset \mathcal K_{\alpha/p, p, \infty} = L^{p, \infty}(\mu_\alpha)$ by H\"{o}lder's inequality and (\ref{embedding2}).
	
Clearly, the $t$ coordinate no longer matters. All we need to prove is that the mapping
$$
u \mapsto u^{N+1} \ast |x|^{-1}
$$
is a contraction (with sufficiently small ratio) on some small neighborhood of zero in $L^{p, \infty}(\mu_\alpha)$.
	
	Indeed, assume $u$ is small in this norm. Raising it to the $N+1$-th power, we get something even smaller in $L^{p/(N+1), \infty}(\mu_\alpha)$. Then by Young's inequality
	$$
	\|f \ast |x|^{-\beta}\|_{L^{p, \infty}(\mu_\alpha)} \les \|f\|_{L^{p/(N+1), \infty}(\mu_\alpha)} 
	$$
	where $\beta$ is required by Proposition \ref{fractional} to be
	$$
	(1-\alpha/3)/p = (1-\alpha/3)(N+1)/p + \beta/3 - 1.
	$$
	But $(1-\alpha/3)N/p = 2/3$, so we get $\beta=1$, which is the value that enables us to close the loop in this quasinorm.
\end{proof}

To interpret Proposition \ref{closed_loop}, take $(\epsilon \phi, \epsilon \psi)$ small bump functions, i.e.\ smooth and compactly supported. We can allow for infinitely many such bumps in the initial data, centered at $(y_j)_{j \in \N}$, as long as for some $\alpha \in (0, \frac {N-2}{N})$
$$
\sup_{j_1} \sum_{j_2 \ne j_1}  \langle y_{j_2}-y_{j_1} \rangle^{-\alpha}<<1.
$$
This is because a bump supported far away from others contributes $|y|^{-\alpha}$ to the quasinorm. Thus
$$
\|\Phi_0(u_0, u_1)\|_{L^p_x(\mu_\alpha) L^\infty_t} \les \epsilon^{1/p} (1 + \sup_{j_1} \sum_{j_2 \ne j_1} \langle y_{j_2}-y_{j_1}\rangle^{-\alpha})^{1/p},
$$
uniformly for $p \in [1, \infty]$.


In other words, fix $(\phi, \psi)$, take $\alpha \in (0, \frac {N-2}{N})$, and $\epsilon \leq \epsilon_0(\alpha)$ (depending on the fractional integration bound, which gets worse as $p \to N+1$ and $\alpha \to \frac{N-2}N$). Consider a sequence $(y_j)_j$ such that
$$
\sup_{j_1} \sum_{j_2 \ne j_1} |y_{j_2}-y_{j_1}|^{-\alpha} < \infty.
$$
One can take for example $y_j = j^{1/\alpha_0} \vec e$ for some fixed vector $\vec e \in \R^3$ and fixed $\alpha_0 < \alpha < \frac {N-2}N$. Then, for sufficiently large $R \geq R_0>>1$, the initial data $(u_0, u_1)$ with bumps centered at $R y_j$
$$
u_0 = \epsilon \sum_j \phi(x-Ry_j),\ u_1 = \epsilon \sum_j \psi(x-R y_j)
$$
yield a small global solution in $L^{p, \infty}_x(\mu_\alpha) L^\infty_t$.

These bumps are asymptotically spaced no closer than $j^{\frac N {N-2}} = j^{\frac 1 {s_c-1/2}}$. In the radially symmetric setting one may perhaps do better: in \cite{loy}, $s_c=3/2$ and the initial data can be taken as a sum of spherical shells of uniform width and height, spaced like $j^{1+\epsilon}$. However, the result is not directly comparable, since initial data are specified on a light cone.

All solutions constructed in this manner are required to have at least one small $L^{p, \infty}_x(\mu_\alpha) L^\infty_t$ norm, where $N+1 < p < 3N/2$, $\alpha = 3-2p/N$. However, the other norms in this family need not be small and can even be infinite (e.g.\ for $\alpha \geq \alpha_0$).

By contrast, the large solutions constructed in Theorem \ref{large_data} are uniformly large in all these norms.

\noindent{\bfseries Future research directions.} In this paper we only used the Choquet and Lorentz--Choquet spaces to provide a response to the referee's remarks. Interesting questions that remain open are:\\
1. What kind of initial data lead to solutions to wave and Schr\"{o}dinger equations in these spaces?\\
2. Do small $L^{p, \infty}_x(\mu_\alpha) L^\infty_t$ solutions preserve regularity? I.e., assuming more regularity (but no extra smallness) for the initial data, can one show that $u \in L^{2N}_{t, x}$?\\
3. Can these norms and quasinorms be used in the study of multisoliton solutions?\\
4. When are these Lorentz-Choquet spaces Banach spaces (i.e.\ when is there a norm equivalent to the quasinorm)?\\
5. For what exponents does Young's inequality hold?\\
6. How do Morawetz and Strichartz inequalities look like in these norms?\\
These and other questions will be addressed in subsequent papers.

\section*{Acknowledgments}
We would like to thank Tom Spencer for the discussions we had on this topic. We also thank the anonymous referee for many interesting and useful remarks.

M.B.\ was partially supported by the NSF grant DMS--1128155, by an AMS--Simons Foundation travel grant, and by a grant from the Simons Foundation (No.\ 429698, Marius Beceanu).

A.S.\ is partially supported by the NSF grants DMS--1201394 and  DMS--01600749 and by a grant from the Simons Foundation (No.\ 395767, Avy Soffer).


\begin{thebibliography}{CKSTT}
\bibitem[Ada]{adams} D.\ R.\ Adams, \emph{Choquet integrals in potential theory}, Publicacions Matem\`{a}tiques, Vol.\ 42 (1998), pp.\ 3–-66.
\bibitem[Bec]{bec_new_schroedinger} M.\ Beceanu, \emph{New estimates for a time-dependent Schr\"{o}dinger equation} Duke Math.\ J.\ (2011) Vol. 159, No.\ 3, pp. 417--477.
\bibitem[BeGo]{becgol} M.\ Beceanu, M.\ Goldberg, \emph{Strichartz estimates and maximal operators for the wave equation in $\R^3$}, Journal of Functional Analysis (2014), Vol.\ 266, Issue 3, pp.\ 1476--1510.
\bibitem[BeL\"o]{bergh} J.\ Bergh, J.\ L\"ofstr\"om, \emph{Interpolation Spaces. An Introduction}, Springer-Verlag, 1976.
\bibitem[BeSo1]{becsof1} M.\ Beceanu, A.\ Soffer, \emph{Large initial data global well-posedness for a supercritical wave equation} (in collaboration with A.\ Soffer), preprint, arXiv:1602.08163.
\bibitem[BeSo2]{becsof2} M.\ Beceanu, A.\ Soffer, \emph{A positivity criterion for the wave equation and global existence of large solutions} (in collaboration with A.\ Soffer), preprint, arXiv:1605.07713.
\bibitem[Bou]{bou} J. Bourgain, \emph{Global wellposedness of defocusing critical nonlinear Schr\"{o}dinger equation in the radial case}, J.\ Amer.\ Math.\ Soc.\ 12 (1999), pp.\ 145--171.
\bibitem[Bul1]{bul1} A.\  Bulut, \emph{The radial defocusing energy-supercritical cubic nonlinear wave equation in $\R^{1+5}$}, preprint, arXiv:1104.2002.
\bibitem[Bul2]{bul2} A.\ Bulut, \emph{Global well-posedness and scattering for the defocusing energy-supercritical cubic nonlinear wave equation}, preprint, arXiv:1006.4168.
\bibitem[Bul3]{bul3} A.\ Bulut, \emph{The defocusing energy-supercritical cubic nonlinear wave equation in dimension five}, preprint, arXiv:1112.0629.
\bibitem[Cho]{choquet} G.\ Choquet, \emph{Theory of capacities}, Ann.\ Inst.\ Fourier (Grenoble) 5 (1953--54), pp.\ 131--295.
\bibitem[Chr]{christo} D.\ Christodoulou, \emph{Global solutions of nonlinear hyperbolic equations for small initial data}, Comm.\ Pure Appl.\ Math.\ 1986, 39, pp.\ 267--282.
\bibitem[CKSTT]{ckstt} J.\ Colliander, M.\ Keel, G.\ Staffilani, H.\ Takaoka, and T.\ Tao, \emph{Global existence and scattering for rough solutions of a nonlinear Schr\"{o}dinger equation on $\R^3$}, Comm.\ Pure Appl.\ Math.\ 57 (2004), pp.\ 987--1014.
\bibitem[DoTh]{doth} Y.\ Do, C.\ Thiele, \emph{$L^p$ theory for outer measures and two themes of Lennart Carleson united}, Bulletin of the AMS, Vol.\ 52, No.\ 2, April 2015, pp.\ 249--296.
\bibitem[DoLa]{dola} B.\ Dodson, A.\ Lawrie, \emph{Scattering for radial, semi-linear, super-critical wave equations with bounded critical norm}, Archive for Rational Mechanics and Analysis, December 2015, Vol.\ 218, Issue 3, pp.\ 1459--1529.
\bibitem[DKM]{dkm} T.\ Duyckaerts, C.\ Kenig, F.\ Merle, \emph{Scattering for radial, bounded solutions of focusing supercritical wave equations}, preprint, arXiv:1208.2158.
\bibitem[DuRo]{duro} T.\ Duyckaerts, T.\ Roy, \emph{Blow-up of the critical Sobolev norm for nonscattering radial solutions of supercritical wave equations on $\R^3$}, preprint, arXiv:1506.00788.
\bibitem[EnMa]{enma} B.\ Engquist, A.\ Majda, \emph{Absorbing boundary conditions for the numerical simulation of waves}, Math.\ Comp.\ 31 (1977), pp.\ 629--651.
\bibitem[GiVe]{give} J.\ Ginibre, G.\ Velo, \emph{Generalized Strichartz inequalities for the wave equation}, J.\ Func.\ Anal., 133 (1995), pp.\ 50--68.
\bibitem[GSV]{gsv} J.\ Ginibre, A.\ Soffer, G.\ Velo, \emph{The global Cauchy problem for the critical non-linear wave equation}, Journal of Functional Analysis (1992), Vol.\ 110, Issue 1, pp.\ 96--130.
\bibitem[KeTa]{keeltao} M.\ Keel, T.\ Tao, \emph{Endpoint Strichartz estimates}, American Journal of Mathematics (1998), Vol.\ 120, No.\ 5, pp.\ 955--980.
\bibitem[KeMe]{keme} C.\ Kenig, F.\ Merle, \emph{Global well-posedness, scattering and blow-up for the energy critical focusing non-linear wave equation}, Acta Mathematica (2008), Vol.\ 201, Issue 2, pp.\ 147--212.
\bibitem[KeMe2]{keme2} C.\ Kenig, F.\ Merle, \emph{Nondispersive radial solutions to energy supercritical non-linear wave equations, with applications}, American Journal of Mathematics (2011), Vol.\ 133, No.\ 4, pp.\ 1029--1065.
\bibitem[GoSc]{goldberg} M.\ Goldberg, W.\ Schlag, \emph{Dispersive estimates for Schr\"{o}dinger operators in dimensions one and three}, Communications in mathematical physics 251 (1) 2004, pp.\ 157--178.
\bibitem[Ker]{ker} S.\ Keraani, \emph{On the blow-up phenomenon of the critical nonlinear Schr\"{o}dinger equation}, J.\ Funct.\ Anal.\ 235 (2006), pp.\ 171--192. MR2216444
\bibitem[KiVi1]{kivi1} R.\ Kilip, M.\ Visan, \emph{The radial defocusing energy-supercritical nonlinear wave equation in all space dimensions}, preprint, arXiv:1002.1756.
\bibitem[KiVi2]{kivi2} R.\ Kilip, M.\ Visan, \emph{The defocusing energy-supercritical nonlinear wave equation in three space dimensions}, preprint, arXiv:1001.1761.
\bibitem[Kla]{klainerman} S.\ Klainerman, \emph{Global existence for nonlinear wave equations}, Commun.\ Pure Appl.\ Math., 33 (1980), pp.\ 43--101.
\bibitem[KlMa]{klma} S.\ Klainerman, M.\ Machedon, \emph{Space-time estimates for null forms and the local existence theorem}, Comm.\ Pure Appl.\ Math, 46 (1993), pp.\ 1221--1268.
\bibitem[KrSc]{krsc} J.\ Krieger, W.\ Schlag, \emph{Large global solutions for energy supercritical nonlinear wave equations on $\R^{3+1}$}, preprint, arXiv:1403.2913.
\bibitem[Li]{li} D.\ Li, \emph{Global wellposedness of hedgehog solutions for the $(3+1)$ Skyrme model}, preprint, arXiv:1208.4977.
\bibitem[LOY]{loy} J.\ Luk, S.-J.\ Oh, S.\ Yang, \emph{Solutions to the Einstein-scalar-field system in spherical symmetry with large bounded variation norms}, preprint, arXiv:1605.03893.
\bibitem[MPY]{mpy} S.\ Miao, L.\ Pei, P.\ Yu, \emph{On classical global solutions of nonlinear wave equations with large data}, preprint, arXiv:1407.4492.
\bibitem[RoSc]{rod} I.\ Rodnianski, W.\ Schlag, \emph{Time decay for solutions of Schr\"{o}dinger equations with rough and time-dependent potentials}, Inventiones mathematicae, March 2004, Vol.\ 155, Issue 3, pp.\ 451--513.
\bibitem[Roy1]{roy} T.\ Roy, \emph{Scattering above energy norm of solutions of a loglog energy-supercritical Schr\"{o}dinger equation with radial data}, preprint, arXiv:0911.0127.
\bibitem[Roy2]{roy2} T.\ Roy, \emph{Global existence of smooth solutions of a 3D loglog energy-supercritical wave equation}, preprint, arXiv:0810.5175.
\bibitem[Str]{struwe} M.\ Struwe, \emph{Global well-posedness of the Cauchy problem for a super-critical nonlinear wave equation in two space dimensions}, Mathematische Annalen 350.3 (2011), pp.\ 707--719.
\bibitem[Tao]{tao} T.\ Tao, \emph{Global regularity for a logarithmically supercritical defocusing nonlinear wave equation for spherically symmetric data}, J.\ Hyperbolic Diff.\ Eq., 4, 2007, pp.\ 259--266.
\bibitem[Tay]{taylor} M.\ E.\ Taylor, \emph{Tools for PDE: Pseudodifferential Operators, Paradifferential Operators, and Layer Potentials}, AMS, Mathematical Surveys and Monographs, Vol.\ 81, 2000.
\bibitem[WaYu]{wang} J.\ Wang, P.\ Yu, \emph{A large data regime for nonlinear wave equations}, Journal of the European Mathematical Society, Vol.\ 18, Issue 3, 2016, pp.\ 575--622.
\bibitem[Yan]{shiwu} S.\ Yang, \emph{Global solutions of nonlinear wave equations with large energy}, preprint, arXiv:1312.7265.

\end{thebibliography}
\end{document}